\documentclass{article}

\usepackage{amsmath}
  \usepackage{paralist}
 \usepackage[colorlinks=true]{hyperref}
\hypersetup{urlcolor=blue, citecolor=red}

\usepackage{amssymb,amsthm,graphicx,stmaryrd,psfrag,pstricks,pslatex,rotating,float}
\usepackage[english]{babel}

\newtheoremstyle{drem}% name
     {3pt}%      Space above
     {3pt}%      Space below
     {\rmfamily}%Body font
     {}%         Indent amount (empty = no indent, \parindent = para indent)
     {\itshape}%Thm head font
     {:}%        Punctuation after thm head
     {.5em}%     Space after thm head: " " = normal interword space
     {}%         Thm head spec (can be left empty, meaning `normal')

\newtheorem{teo}{Theorem}[section]
\newtheorem{lem}[teo]{Lemma}
\newtheorem{prop}[teo]{Proposition}
\newtheorem{cor}[teo]{Corollary}
\newtheorem{defi}[teo]{{\rm Definition}}
\newtheorem{ques}[teo]{{\rm Question}}
\theoremstyle{drem}
\newtheorem{rem}[teo]{Remark}
\newtheorem{ex}[teo]{Example}

\let\oldmarginpar\marginpar
\renewcommand\marginpar[1]{\-\oldmarginpar[\raggedleft\footnotesize #1]%
{\raggedright\footnotesize #1}}

\def\og{\leavevmode\raise.3ex\hbox{$\scriptscriptstyle\langle\!\langle$~}}
\def\fg{\leavevmode\raise.3ex\hbox{~$\!\scriptscriptstyle\,\rangle\!\rangle$}}

\def\vs{\vrule width 0cm height 0.1in depth 0in}
\newcommand{\fr}[2]{\frac{\displaystyle\vs #1}{\displaystyle\vs #2}}

\newcommand{\ie}[0]{\emph{i.e. }}
\newcommand{\eg}[0]{\emph{e.g. }}
\newcommand{\cf}[0]{\emph{cf. }}

\newcommand{\eps}[0]{\epsilon}

\newcommand{\bv}[0]{\! \big|}

\newcommand{\vide}[0]{\varnothing}
\renewcommand{\setminus}{\smallsetminus}

\newcommand{\un}[0]{1 \!\! \mathrm{l}} %\mathrm{I} }
\newcommand{\del}[0]{\partial \!}

\newcommand{\rr}[0]{\mathbb{R}}
\newcommand{\zz}[0]{\mathbb{Z}}
\newcommand{\nn}[0]{\mathbb{N}}
\newcommand{\kk}[0]{\mathbb{K}}

\newcommand{\srl}[1]{\overline{#1}}

\newcommand{\jo}[1]{\mathcal{#1}}

\newcommand{\limm}[1]{\textrm{\raisebox{.5ex}{\mbox{$\underset{#1}{\lim}$}}} \:}
\newcommand{\lims}[1]{\textrm{\raisebox{.5ex}{\mbox{$\underset{#1}{\limsup}$}}} \:}

\newcommand{\maxx}[1]{\textrm{\raisebox{.5ex}{\mbox{$\underset{#1}{\max}$}}} \:}
\newcommand{\minn}[1]{\textrm{\raisebox{.5ex}{\mbox{$\underset{#1}{\min}$}}} \:}
\newcommand{\supp}[1]{\textrm{\raisebox{.5ex}{\mbox{$\underset{#1}{\sup}$}}} \:}
\newcommand{\inff}[1]{\textrm{\raisebox{.5ex}{\mbox{$\underset{#1}{\inf}$}}} \:}
\newlength{\entrh}   %pour les inf près des sup
\newcommand{\somme}[2]{\overset{#2}{\underset{#1}{\sum}}}
\newcommand{\Somme}[2]{\overset{#2}{\underset{#1}{\displaystyle\vs \sum}}}

\newcommand{\union}[2]{\overset{#2}{\underset{#1}{\cup}}}

\newcommand{\nr}[1]{\left\| #1 \right\|}
\newcommand{\pnr}[1]{\| #1 \|}

\newcommand{\gen}[1]{\left\langle #1 \right\rangle}
\newcommand{\pgen}[1]{\langle #1 \rangle}

\newcommand{\ssi}[0]{\Leftrightarrow}
\newcommand{\imp}[0]{\Rightarrow}
\newcommand{\inj}[0]{\hookrightarrow}

\newcommand{\diam}[0]{\mathrm{Diam}\,}
\newcommand{\tr}[0]{\mathrm{Tr}\,}
\newcommand{\img}[0]{\mathrm{Im}\,}
\newcommand{\Id}[0]{\mathrm{Id}}

\newcommand{\homo}[0]{\mathrm{Hom}}

\newcommand{\IN}[0]{\mathrm{int}}
\newcommand{\FE}[0]{\mathrm{clo}}

\newcommand{\wdm}[0]{\mathrm{wdim} \,}
\newcommand{\ldm}[0]{\mathrm{ldim} \,}

\newcommand{\ev}[0]{\, e \! v}

\newcommand{\dlp}[0]{\mathrm{dim}_{\ell^p}}
\newcommand{\dl}[1]{\mathrm{dim}_{\ell^{#1}}}

\newcommand{\ddlp}[0]{\mathrm{dim}^*_{\ell^p}}
\newcommand{\ddl}[1]{\mathrm{dim}^*_{\ell^{#1}}}

\newcommand{\mydim}[0]{\mathrm{dim}_\mathrm{idl}}
\newcommand{\mysp}[0]{\mathfrak{X}}

\newcommand{\pn}[1]{\llparenthesis #1 \rrparenthesis}

\newcommand{\pnc}{\llparenthesis \cdot \rrparenthesis}

\newcommand{\yper}[0]{\Pi}
\newcommand{\yperf}[0]{\overline{\Pi}}

\newcommand{\comp}[0]{\mathsf{c}}

\newcommand{\perpud}[0]{{\, \textrm{\mbox{\raisebox{.7ex}{\:\begin{rotate}{180}\makebox(0,0){$\perp$}\end{rotate}}}} \,}}

\newcommand{\maz}[0]{\mu_\circ}

\title{Further properties of $\ell^p$ dimension}

\author{Antoine Gournay}

\begin{document}
\maketitle

% Enter the first author's name and address:
\centerline{\scshape Antoine Gournay }
\medskip
{\footnotesize
% please put the address of the first author
 \centerline{ Universit\'e de Neuch\^atel,}
   \centerline{Rue \'E.-Argand, 11,}
   \centerline{2000 Neuch\^atel, Switzerland}
} % Do not forget to end the {\footnotesize by the sign }

\medskip

% The name of the associate editor will be entered by an editorial staff
% \centerline{(Communicated by the associate editor name)}

%checker edmunds evans spectral theory ... p90 prop 3.4 et autours

\begin{abstract}
This article establishes more properties of the $\ell^p$ dimension introduced in previous article. That is given an amenable group $\Gamma$ acting by translation on $\ell^p(\Gamma)$, a number satisfying dimension-like properties is associated to the (usually infinite dimensional) subspaces $Y$ of $\ell^p(G)$ which are invariant under the action of $G$. As a consequence, for $p\in [1,2]$, if $Y$ is a closed non-trivial $\Gamma$-invariant subspace of $\ell^p(\Gamma;V)$ and let $Y_n$ is an increasing sequence of closed $\Gamma$-invariant subspace such that $\srl{\cup Y_n} = \ell^p(\Gamma;V)$, then there exist a $k$ such that $Y_k \cap Y \neq \{0\}$.
\end{abstract}

%-------------------------------------------------------------------
\section{Introduction}\label{sintro}
%-------------------------------------------------------------------

\par This article continues previous work where an attempt was made to produce an $\ell^p$ version of the von Neumann dimension. To present things in their simplest incarnations, let $V$ be a finite dimensional normed vector space (over a field $\kk$, typically $\mathbb{C}$ or $\rr$) and $\Gamma$ a countable (infinite) amenable group (\eg any Abelian group, such as $\zz$). Let $\ell^p(\Gamma;V)$ be the Banach space of maps $f:\Gamma \to V$ such that $\sum_{\gamma \in \Gamma} \|f(\gamma)\|^p$ is finite (\eg if $\Gamma = \zz$, then these are double-ended $p$-summable sequences). There are also two other so-called classical Banach spaces of interest here: $\ell^\infty(\Gamma;V)$, the space of bounded functions, and $c_0(\Gamma;V)$ the space of bounded functions decreasing to $0$ at infinity (more precisely, $f \in c_0(\Gamma;V)$ if for any increasing sequence of finite sets $F_i$ with $\cup F_i = \Gamma$, the associated sequence of real numbers $\supp{\gamma \notin F_i} \| f(\gamma)\|$ tends to $0$).

\par There is a natural of $\Gamma$ action on $\ell^p(\Gamma;V)$ given by translation: let $f \in \ell^p(\Gamma;V)$ then $(\gamma \cdot f)(\gamma') = f(\gamma^{-1} \gamma')$. The subject matter of this article are the linear subspaces $X$ of $\ell^p(\Gamma;V)$ such that $X$ is invariant under the action of $\Gamma$: $\forall \gamma \in \Gamma, \gamma X = X$. The goal is to show that it is possible to associate to such linear subspaces a dimension, that is, a positive  real number (most of the time not an integer) which behaves nicely under common operations of vector spaces (see the list of properties in section \ref{sprop}). Though not crucial to the work, it will also be assumed that $\Gamma$ is finitely generated.

\par This work is motivated by problems in the $\ell^p$ cohomology of groups, and, in particular, the following question (due to D.~Gaboriau):

\begin{ques}\label{laquestion}
Let $Y$ be a closed non-trivial $\Gamma$-invariant subspace of $\ell^p(\Gamma;V)$ and let $Y_n$ be an increasing sequence of closed $\Gamma$-invariant subspace such that $\srl{\cup Y_n} = \ell^p(\Gamma;V)$. Does there exist a $k$ such that $Y_k \cap Y \neq \{0\}$?
\end{ques}

\par It is quite easy to solve this question positively (\ie such a $k$ does exist) if $p=2$. The arguments known to the author however involve consciously or not the use of von Neumann dimension, a Hilbertian concept. As a simple consequence to the properties of $\dlp$, this question admits a positive answer for $p \in [1,2]$.

\begin{teo}\label{lareponse}
Let $p \in [1,2]$. Let $Y$ be a closed non-trivial $\Gamma$-invariant subspace of $\ell^p(\Gamma;V)$ and let $Y_n$ be an increasing sequence of closed  $\Gamma$-invariant subspace such that $\srl{\cup Y_n} = \ell^p(\Gamma;V)$. There exist $k \in \nn$ such that $Y_k \cap Y$ is a non-trivial $\Gamma$-invariant linear subspace (of positive $\ell^p$-dimension and infinite usual dimension).
\end{teo}

\par This result should make a proof of the vanishing of the reduced $\ell^p$ cohomology, $p \in ]1,2[$ for amenable groups more tractable, keeping in mind  Gaboriau's proof in the $\ell^2$ setting, \cite{Gab}. Remember that the reduced $\ell^1$ cohomology is non-trivial, even for $\Gamma = \zz$, see \cite[Example 3 in \S{}4]{MV}. However, the question \ref{laquestion} admits a very straightforward proof for any discrete group if $p=1$, see remark \ref{repl1}. 
% Here is another consequence of some of the results presented here (see section \ref{scohom}):
% 
% \begin{teo}\label{laballe}
% Let $p \in [2,\infty[$ and $\Gamma$ be a discrete finitely generated amenable group. Then there are no non-constant $p$-harmonic funtions on $\Gamma$, in particular, the reduced rank one $\ell^p$-cohomology of $\Gamma$ is trivial.
% \end{teo}

\par This article starts by presenting the desired properties of an $\ell^p$ dimension and how those established yield theorem \ref{lareponse}. Preliminary results on amenable groups and the definition of the $\ell^p$ dimension are given in section \ref{sprel}. Section \ref{sdual} describes the dual $\ell^p$ dimension and how its properties can be inferred from any pre-existing notion of $\ell^p$ dimension (\ie another definition of $\ell^p$ dimension could very well be used in this section). The remaining sections concentrate on the proofs of the various properties. 
% Section \ref{spos-f} discusses completion and it is shown that non-trivial spaces have positive dimension for any $p \neq \infty$ and in $c_0$, but the actual value is not explicitly bounded. 
Section \ref{spos} gives more explicit lower bounds for positivity, given that $p \in [1,2]$, which allow, in section \ref{scont}, to establish one of the continuity properties. Completion and maximality are also discussed in section \ref{scont}. Section \ref{sadd} establishes additivity, reduction and reciprocity for any $p$. Section \ref{sinv} presents some results on invariance. 
%Section \ref{skerim} treats the case of spaces which are the kernel or the image of a map of finite type.
Section \ref{sfin} gives some concluding remarks and questions. Out of concern for completion, some of the proofs are presented again here (though they offer either different point of view, corrections or sometimes refinement of the previous results). 

\par As the author was preparing this work, G.\'Elek mentioned to him the work of B.Hayes \cite{Hay} which also defined a $\ell^p$ dimension for a (much) more general class of groups known as sofic groups and shows more properties than in \cite{moi-lp}. However, neither the properties presented in the present text nor the methods seem to be completely redundant, even for amenable groups. It might not be completely trivial to check that the two definitions are equal (see paragraph after question \ref{q:egal}). There is also an article of D.~Voiculescu \cite{Voi} which gives an entropy-like approach (but only in the Hilbertian case).

% \emph{Remark:} \cite[Theorem 1.1]{moi-lp} can actually be obtained without restricting the attention to maps of finite type, with much less effort and for any discrete group $\Gamma$. Indeed, as was pointed out by V.~Lafforgue, the spaces $\mathrm{End}_{p,V}$ of endomorphisms of $\ell^p(\Gamma;V)$ are quite different depending on the dimension of $V$. More precisely, some relations hold in some of these spaces but not others (\eg when $\dim V =1$ commutativity is automatic, but in higher dimension matrix are no longer commutative). As such, there can be no isomorphism between $\ell^p(\Gamma;V)$ and $\ell^p(\Gamma;V')$ as it would induce an isomorphism on these algebraically non-isomorphic spaces. 

\emph{Apology:} There is a mistake in \cite[Corollary 3.9]{moi-lp}. Two ingredients are used in this the proof which are not mentioned in the statement: the image has to be closed (so that the inverse is bounded), and $\ell^\infty$ should be excluded (so that functions of finite support are norm-dense). In the said the paper, invariance is done for maps of finite type and closed image and the ambient space should differ from $\ell^\infty$. See example \ref{cexinv} for a counter-example in $\ell^\infty$ and theorem \ref{tinv-new} for a proof of P2 which works under less restrictive hypothesis.

\emph{Acknowledgements:} Conversations with G.~\'Elek, D.~Gaboriau, B.~Hayes, B.~Nica, and N.~Raymond were benefical to this work and they are accordingly warmly thanked.
% O.~Besson, B.~Hayes, P.~Pansu, A.~Valette 

%---------------------------------------------------------------------------------
\section{Properties} \label{sprop}
%---------------------------------------------------------------------------------

\par All properties listed in this section should be read as tentative. They are a wish-list for the properties of an ideal dimension (following the work of Cheeger and Gromov \cite{CG}, and keeping in mind question \ref{laquestion}). The current state of these properties is summarized in theorem \ref{propvrai}. In the following statements $\mysp$ is a (unspecified among the classical Banach spaces) space of functions from $\Gamma$ to $V$ endowed with a ($\Gamma$-invariant) norm and $\mydim$ is a map from $\Gamma$-invariant subspaces of $\mysp$ in the positive real numbers.
\begin{enumerate}\renewcommand{\labelenumi}{{\normalfont P\arabic{enumi}}}
\item (Normalization) $\mydim \mysp(\Gamma;V) =\dim_\kk V$;
\item (Invariance) If $f:Y_1 \to Y_2$ is an injective $\Gamma$-equivariant linear (continuous) map of finite type, then $\mydim Y_1 \leq \mydim Y_2$;
\item (Completion) If $\srl{Y}$ is the norm completion of $Y$ in $\mysp(\Gamma;V)$, then $\mydim Y =\mydim \srl{Y}$;
\item (Reduction) If $\Gamma_1 \subset \Gamma_2$ is a subgroup of finite index, and if $Y \subset \mysp(\Gamma_2;V)$ is seen by restriction as a subspace of $\mysp(\Gamma_1;V^{[\Gamma_2:\Gamma_1]})$ then $[\Gamma_2:\Gamma_1] \mydim (Y,\Gamma_2) = \mydim (Y,\Gamma_1)$;
\item (Reciprocity) If $\Gamma_1 \subset \Gamma_2$ is a subgroup of finite index and if $Y_2 \subset \mysp(\Gamma_2;V)$ is the subspace induced by $Y_1 \subset \mysp(\Gamma_1;V)$ then $\mydim (Y_2,\Gamma_2) = \mydim (Y_1,\Gamma_1)$.
\item (Additivity) $\mydim Y_1 \oplus Y_2 = \mydim Y_1 + \mydim Y_2$;
\item (Positivity) $Y \subset \mysp$ is trivial if and only if $\mydim Y =0$.
\item (Right-continuity) If $\{Y_i\}$ is a decreasing sequence of closed linear subspaces then $\mydim (\cap Y_i) = \limm{i \to \infty} \mydim Y_i$;
\item (Left-continuity) Let $Y_n$ be an increasing sequence of subspaces and let $Y=\overline{\cup Y_n}$ then $\limm{n \to \infty} \mydim Y_n = \mydim Y$.
\item (Maximality) If $Y\subset \mysp(\Gamma;V)$ is a closed subspace and $\mydim Y = \dim_\kk V$ then $Y$ is of finite codimension in $\mysp(\Gamma;V)$.
\end{enumerate}
\par The interest of the last property comes from the dual $\ell^p$ dimension, see section \ref{sdual}. Here, finite codimension means that the vector space $\mysp / Y$ is of finite dimension. Maps of finite type (in P2) are defined in definition \ref{dtype}, and the induced subspace (in P5) is defined in  definition \ref{dinduc}. Note that the fininiteness of codimension is superfluous in \emph{P10} when $p\in ]1,\infty[$ or in $c_0$. Indeed, is a closed space has finite codimension then $Y^\perp$ is finite dimensional (which is impossible unless it is trivial).

\par In view of Question \ref{laquestion}, it seems that P8 (right-continuity) is a stronger type of continuity than what is required. If one thinks of spaces ordered by inclusion the smallest being on the ``left'', then P8 amounts to for continuity on the right. However, the right type of continuity to answer question \ref{laquestion} is left-continuity. Actually, weaker versions of continuity will be sufficient, namely

\begin{enumerate}\renewcommand{\labelenumi}{{\normalfont P\arabic{enumi}'}}\setcounter{enumi}{7}
\item (Right-continuity at $\{0\}$)  Let $Y_n$ be an decreasing sequence of closed subspaces such that $\cap Y_n = \{0\}$ then $\limm{n \to \infty} \mydim Y_n = 0$. 

\item (Left-continuity at $\mysp$) Let $Y_n$ be an increasing sequence of subspaces such that $\overline{\cup Y_n} = \mysp(\Gamma;V)$ then $\limm{n \to \infty} \mydim Y_n = \dim_\kk V$.

\item (Weak$^*$-maximality) If $Y\subset \mysp(\Gamma;V)$ is a weak$^*$-closed subspace and $\mydim Y = \dim_\kk V$ then $Y$ is $\mysp(\Gamma;V)$.
\end{enumerate}

\par The last property will come out naturally when looking at the dual version of a dimension. Note that it is superfluous in the case $p\in ]1,\infty[$ or $c_0$. Indeed, is a closed space has finite codimension then so does its dual, but $(\ell^p /Y)^* = Y^\perp$ which cannot be finite dimensional unless it is trivial. It is quite standard that all these properties hold for $\mysp = \ell^2$ and $\mydim = \dl{2}$ (as it is equal to the von Neumann dimension, see \cite[Corollary A.2]{moi-lp}). Here are the main statements of this article.

\begin{teo}\label{propvrai}
In the statements above take $\mydim$ to be $\dlp$ and $\mysp$ to be $\ell^p$. 

\par Properties \emph{P1-P6} hold for any $p \in  [1,\infty[$. Properties \emph{P7} and \emph{P9'} are true for $p \in [1,2]$. \emph{P7} cannot be true if $p>2$. If $p \in ]2,\infty[$ then Property P10 also holds. 

\par If  $p=1$, then  \emph{P8} and \emph{P8'} fail. If $p=\infty$ then \emph{P1}, \emph{P3-P6} hold but \emph{P2}, \emph{P7}, \emph{P9} and \emph{P10} fail.

However, if $\mydim$ is $\dl{\infty}$ and $\mysp$ is $c_0$ then \emph{P1-P6} and \emph{P10} hold.
\end{teo}
\begin{tabular}{l|cccccccccc}
  & P1& P2& P3& P4& P5& P6& P7&P8'&P9'&P10\\
\hline
$\dl{1}$ in $\ell^1$ 
  & T & T & T & T & T & T & T & F & T & ? \\
$\dlp$ in $\ell^p$ for $p \in ]1,2[$ 
  & T & T & T & T & T & T & T & ? & T & ?\\
$\dlp$ in $\ell^p$ for $p \in ]2,\infty[$
  & T & T & T & T & T & T & F & ? & ? & T\\
$\dl{\infty}$ in $c_0$  
  & T & T & T & T & T & T & F & ? & ? & T\\
$\dl{\infty}$ in $\ell^\infty$
  & T & F & T & T & T & T & F & ? & F & F
\end{tabular}
\begin{teo}\label{propvraidual}
Let $\mydim$ be $\ddlp$ and $\mysp$ to be $\ell^p$. 

\par Properties \emph{P1-P6} hold for any $p \in ]1,\infty[$. Properties \emph{P8'} and \emph{P10} are true for $p \in [2,\infty[$. \emph{P10} cannot be true for $p\in]1,2[$. If $p=1$, then only \emph{P1} and \emph{P3-P6} are known to hold. \emph{P7} holds if $p\in]1,2[$. 

Furthermore, if one takes $\mydim$ to be $\ddl{\infty}$ and $\mysp$ to be $c_0$, then \emph{P1-P6}, \emph{P8'} and \emph{P10} hold.
\end{teo} 

\begin{tabular}{l|cccccccccc}
  & P1& P2& P3& P4& P5& P6& P7&P8'&P9'&P10\\
\hline
$\ddl{1}$ in $\ell^1$ 
  & T & ? & T & T & T & T & T & F & ? & F \\
$\ddlp$ in $\ell^p$ for $p \in ]1,2[$ 
  & T & T & T & T & T & T & T & ? & ? & F\\
$\ddlp$ in $\ell^p$ for $p \in ]2,\infty[$
  & T & T & T & T & T & T & ? & T & ? & T\\
$\ddl{\infty}$ in $c_0$  
  & T & T & T & T & T & T & ? & T & ? & T
\end{tabular}\\[.2cm]

\par Regarding $\dlp$, many of the above statements are found in \cite{moi-lp}: properties P1-P4 are established there, a counterexample \cite[Example]{moi-lp} is given to P8 (right-continuity) and P8' if $p=1$, while P7 (positivity) obviously fails for $p=\infty$ (look at $Y$ generated by a periodic function and its translates). The author apologizes as the proof of P2 given there only works for maps of finite type and closed image, and does not work in $\ell^\infty$ (see example \ref{cexinv}).

\par In the present text, P6 for $\dlp$ is theorem \ref{tadd}, proposition \ref{tredu} covers P4, P5 is corollary \ref{trec}, P7 is theorem \ref{tpos}, the counterexample to P9 and P10 for $p= \infty$ is example \ref{cexlcont}, %P8' is theorem \ref{tcont0}, 
P9' is the content of theorem \ref{tcontlp}, the counterexample to P2 in $\ell^\infty$ is example \ref{cexinv}, a proof of P2 in the other cases is theorem \ref{tinv-new} (see also proposition \ref{tinvtypl1} for other results), and proposition \ref{tcomp} covers P3. Maximality P10 is done in propositions \ref{maxlp+2}.% and \ref{maxlp-2}.

\par The dual version, $\ddlp$, is defined in section \ref{sdual} and its properties are also proved there (they are basically consequences of the properties of $\dlp$).

\par If one would be interested in adapting the arguments of \cite{CG}, the actual properties needed are P1-P8 (P5 is actually stated there for infinite index subgroups). Whereas, as shall be seen shortly, a positive answer to question \ref{laquestion} requires only P1, P2, P6, P7 and P9'.

\begin{proof}[Proof of Theorem \ref{lareponse}] (Recall that $\Gamma$ amenable and $p \in [1,2]$.)
Let $d = \dlp H$ then by P7, $d>0$. Since $\cup H_n = \ell^p(\Gamma;V)$ then, by P9', there exists a $k$ with $\dlp H_k > \dim_\kk V - d/2$. Using P6, $\dlp H_k \oplus H > \dim_\kk V$. If $H_k \cap H = \{0\}$ then there is a linear injective map (of finite type, continuous since $H_k$ and $H$ are closed) $H_k \oplus H \to \ell^p(\Gamma;V)$. Thanks to P1 and P2,  $\dlp H_k \oplus H \leq \dim_\kk V$, a contradiction. Thus $H_k \cap H \neq \{0\}$. 
\end{proof}

%----------------------------------------------------------------------------
\section{Definitions and countable amenable groups} \label{sprel}
%----------------------------------------------------------------------------

Perhaps the first notational commodity which should be mentioned is that $p'$ is the conjugate exponent of $p$, \ie $p' = \tfrac{p}{p-1}$. Also, when $\Omega \subset \Gamma$, $\Omega^\comp = \Gamma \setminus \Omega$ is the (set)-complement of $\Omega$. In order to make the other definitions, some basic notions on (finitely generated) countable amenable groups will be made.
\begin{defi}\label{dmoy}
For $\Gamma$ a finitely generated (countable) group, let $F, \Omega \subset \Gamma$ then the $F$\emph{-boundary of} $\Omega$ is
\[
\del_F   \Omega
   = \{\gamma \in \Gamma | \gamma F \cap \Omega \neq \vide \textrm{ and } \gamma F \cap \Omega^c \neq \vide\}
\]
When these sets are finite, the {\em relative amenability function} is $\alpha(\Omega;F) = \frac{|\del_F\Omega|}{|\Omega|}$. The group $\Gamma$ is {\em amenable} if there exists a sequence $\{ \Omega_i \}_{i\geq 1}$ of finite subsets of $\Gamma$ such that $\forall F \subset \Gamma$ finite, $\limm{i \to \infty} \alpha(\Omega_i;F) = 0$. Such a sequence $\{ \Omega_i \}$ is called a F{\o}lner sequence for $\Gamma$.
\end{defi}

\par The definition of $\dlp$ is based on a notion of ``thick'' dimension of some spaces. A pseudo-norm $\pnc$ is a norm except that $ \pn{x} =0 \nRightarrow x=0$. For $X \subset V$ the notation $\diam X = \supp{x,x' \in X} \pn{x-x'}$ will be maintained (though obviously a set of diameter $0$ can be bigger than a point). That said, the following definition is essentially a reversion of quantifier in the definition of Kolmogorov's width. 

\begin{defi}\label{dldm}
Let $V$ be a vector space with a pseudo-norm $\pnc$ and let $X \subset V$ be a subset. Define $\ldm_\eps (X, \pnc)$ as the smallest dimension of a vector space $V'$ such that there exists a linear map $f:V \to V'$ with $\diam (\ker f \cap X) \leq \eps$. 
\end{defi}

\par Remark that the above definition is equivalent to looking for the smallest codimension $k$ of a linear subspace $L^{-k}$ of $V$ such that $\diam (L^{-k} \cap X) \leq \eps$. When $X$ is convex and centrally symmetric (around the origin) this is nothing but looking for a $L^{-k}$ such that the biggest pseudo-norm of an element $x$ in $L^{-k} \cap X$ is less than $\tfrac{\eps}{2}$.

\begin{ex}\label{ebase}
Let $X$ be a normed vector space with the pseudo-norm being actually the norm $\nr{x-x'}$. Let $A = B^X_1$ be its unit ball. Then $\ldm_\eps (A,\pnr{\cdot}) = \dim_\kk X$ if $\eps < 2$ (if the map has a non-trivial kernel then it will contain two opposite points on the sphere) and  $\ldm_\eps (A,\pnr{\cdot}) = 0$ if $\eps \geq 2$ (consider the map which sends all of $A$ to one point). 

\par As a consequence, one can easily see that $\ldm_\eps (B^{\ell^p(\Gamma;V)}, \ev_{\ell^p(\Omega)}) = |\Omega| \dim_\kk V$. This is actually a proof of P1 for $\dlp$.
\end{ex}

\par If one considers throughout \cite{moi-lp} the category subsets in vector spaces endowed with a pseudo-norm and linear maps (rather than a topology, a pseudo-distance, and continuous maps) then the same results (upon restricting to the said category) can be obtained by replacing Urysohn's widths (\ie $\wdm$) with Kolmogorov's widths (\ie $\ldm$).

\par Trivially, $\wdm_\eps X \leq \ldm_\eps X$, but, more importantly, Donoho points out in \cite[Section III.D]{Dono} the following inequality: $\ldm_\eps X \leq \wdm_{\eps /2^{1+1/p}} X$. So the definition below is equivalent to the definition in \cite{moi-lp}.

\begin{defi} \label{defdlp}
Let $V$ be a finite-dimensional normed vector space. Let $Y \subset \ell^\infty(\Gamma;V)$ be a subspace invariant by the natural left action of $\Gamma$, an amenable countable group. Let $\{ \Omega_i \}_{i\geq 1}$ be a F{\o}lner sequence for $\Gamma$. Then, for $p\in [1,\infty]$, the $\ell^p$ dimension of $Y$ is defined by
\[
\dlp (Y, \{\Omega_i\}) = \limm{\eps \to 0} \lims{i \to \infty} \frac{\ldm_\eps (B^{Y}_1, \ev_{\ell^p(\Omega_i)})} {|\Omega_i|}
\]
where $B^{Y}_1 = Y \cap B^{\ell^p(\Gamma;V)}_1$, 
and $\ev_{\ell^p(\Omega_i)}(y,y') = \| y-y' \|_{\ell^p(\Omega_i)} = \Big( \somme{\gamma \in \Omega_i}{} \|y(\gamma)-y'(\gamma) \|_V^p \Big)^{1/p}$ if $p\neq \infty$, while $\ev_{\ell^\infty(\Omega_i)}(y,y') = \| y-y' \|_{\ell^\infty(\Omega_i)} = \supp{\gamma \in \Omega_i} \|y(\gamma)-y'(\gamma) \|_V$. 
\end{defi}

\par It is important that $B^{Y}_1$ is not the ball for the pseudo-norm $\ev_{\ell^p(\Omega)}$; it is the intersection of $Y$ with the ball of radius $1$ in $\ell^p(\Gamma)$ (which is the same thing as the unit ball of $Y$, when $Y \subset \ell^p(\Gamma;V)$). 

\par The above definition depends on a F{\o}lner sequence. It is known that $\dlp (Y, \{\Omega_i\})$ is actually independent of the choice of F{\o}lner sequence $\{\Omega_i\}$ (see \cite[Corollary 5.2]{moi-lp}). Furthermore, if $Y \subset \ell^2(\Gamma;V)$, $\dl{2} Y$ coincides with the von Neumann dimension (see \cite[Corollary A.2]{moi-lp}). Lastly, remark that if $\Gamma$ were to be finite, then $\dlp Y$ would simplify to $\frac{\dim Y}{|\Gamma|}$.

\par Among the results of amenable groups, a lower bound on the number of translates of a finite set $F$ in the sets of the F{\o}lner sequence will be used. Let $\Omega$ be another finite set. Define $\beta(\Omega;F)$ to be the maximal number of translates of $F$ that can be packed (without intersection) in $\Omega$.

\par To estimate this quantity, and also for other results of this paper, a small number of lemmas and definitions must be made. As our focus is countable groups, the (Haar) measure is always the counting measure, denoted $|\cdot|$. Many things remain true for non-countable groups, the curious reader is referred to the paper of Ornstein and Weiss \cite{OW}.

\begin{defi}\label{dqt}
Let $\eps \in]0,1[$. Finite subsets $\{F_i\}_{1 \leq i \leq n}$ of $\Gamma$ will be said $\eps$-disjoint if there exists $F'_i \subset F_i$ which are disjoint and such that $|F_i'| \geq (1-\eps)|F_i|$ and $\cup F_i' = \cup F_i$.
\par A finite subset $\Omega$ will be said to admit an $\eps$-quasi-tiling by the subsets $\{F_i\}_{1 \leq i \leq n}$ if
\begin{enumerate}\renewcommand{\labelenumi}{{\normalfont (\alph{enumi})}}
\item $F_i \subset \Omega$,
\item the $F_i$ are $\eps$-disjoint,
\end{enumerate}
\end{defi}

\par The proof of the following lemmas can be found in \cite{OW} (in their original form), \cite{Kri} or \cite[Section 5]{moi-lp}.

\begin{lem}\label{epac}
Let $\Gamma$ be a countable group. Let $\Omega\subset \Gamma$ and $e_\Gamma \in F \subset \Gamma$ both finite sets and such that $\alpha(\Omega;F) < 1$. 
Let $G$ be a maximal (with respect to inclusion) finite subset of $\Gamma$ such that the $\{\gamma F\}_{\gamma \in G}$ form an $\eps$-quasi-tiling of $\Omega$. Let $U_F = \union{\gamma \in \Gamma}{} \gamma F$, then

%Let $\{\gamma_i\}_{1\leq i \leq k}$ be a maximal sequence of elements of $\Gamma$ such that the $\gamma_i F$ form an $\eps$-quasi-tiling of $\Omega$. Let $U_F^i = \union{j=1}{i} \gamma_j F$, then
\[
\fr{|U_F|}{|\Omega|} \geq \eps (1-\alpha(\Omega;F)).
\]
\end{lem}

\begin{lem}\label{paciter}
Let $\{F_i\}_{i \in \zz_{>0}}$ be a F{\o}lner sequence for $\Gamma$, let $\delta \in ]0,1/2[$. Then there exists an integer $n(\delta)$, a \emph{finite} subsequence (which depends on $\delta$) $\{F_{i_k}\}_{1 \leq k \leq n}$,  and finite subsets $G_k \subset \Gamma$ such that for all set $\Omega$ containing a translate of $F_{i_n}$ and satisfying $\alpha(\Omega; F_{i_n}) \leq 2\delta^{2n}$ there exists a  family $\jo{G}$ of $\delta$-disjoint sets such that $|\union{F \in \jo{G}}{} F | \geq (1-\delta)|\Omega|$ and $\jo{G}$ consists in $|G_k|$ translates of each $F_{i_k}$.
\end{lem}

\par The next lemma is the promised estimate on the ``packing'' of finite sets. Lemma \ref{epac} and \ref{paciter} are the only two necessary ingredients.

\begin{lem}\label{tcardtil}
Let $F$ be a finite subset of $\Gamma$, and let $\Omega_i$ be a F{\o}lner sequence. There exists $i_F \in \nn$ such that for all $i \geq i_F$ and for any finite set $G \subset \Gamma$ maximal (under inclusion) among those who satisfy that $\{\gamma F\}_{\gamma \in G}$ is a disjoint collection of set inside $\Omega_i$. Furthermore, 
\[
\frac{|\Omega_i|}{|F|} \geq |G_i| %= \beta(\Omega_i;F) 
\geq \frac{ \big( 1 - \alpha(\Omega_i;F) \big) |\Omega_i|}{|F|^2}
\]
%and $\forall k\neq k', \gamma_k F \cap \gamma_{k'} F = \vide$ and $\gamma_k F \subset \Omega$. 
In particular, $|F|^{-1} \geq \lim_i \beta(\Omega_i;F)/|\Omega_i| \geq |F|^{-2}$ and the limit is independent of the choice of sequence.
\end{lem}

\begin{proof}
This is a standard result for discrete amenable groups. The independence on the choice of sequence for the limit is a consequence of the Ornstein-Weiss lemma applied to the function $\Omega \mapsto \beta(\Omega;F)$ (see \cite{OW}, \cite{Kri} or \cite[Theorem 5.1]{moi-lp}). The upper bound $|G|  \leq |\Omega|/|F|$ is an obvious cardinality argument. 
\par The lower bound comes from the existence of $\eps$-quasi-tilings (see definition \ref{dqt} and lemma \ref{epac} above or \cite[Lemma 5.5]{moi-lp}). Pick $\eps = \frac{1}{(1+\eps_2)|F|}$ (for some arbitrarily small $\eps_2 >0$) in said lemma to produce a disjoint quasi-tiling. Remark then that 
%$k$ corresponds to a subset of cardinality $\beta(\Omega;F)$ and that 
$|U_F| = |G| \, |F|$ to get $|G| \, |F| \geq \big(1-\alpha(\Omega;F)\big) |\Omega| / |F|$.
\end{proof}

Before closing this introduction, we make a small parenthesis about $\Gamma$-equivariant maps. Though there is definitively still a good measure of freedom left (see Derighetti's book \cite[\S{}1.2]{Der}), most of the maps are of a peculiar type. %Recall that if $V_1$ and $V_2$ are two normed vector spaces, then there is a natural norm on $\homo(V_1,V_2)$.

\begin{lem}\label{ttypl1}
Let $\mysp$ be $\ell^p$ for $1 \leq p < \infty$ or $c_0$ and let $q \in [1,\infty]$. 
If $f: \mysp(\Gamma;V_1) \to \ell^q(\Gamma;V_2)$ is $\Gamma$-equivariant, then $f$ admits a formal expression as a (right-)convolution. %is of $\ell^r$ type where $r^{-1} \geq q^{-1} + p'^{-1}$ (in particular, if $q=p$ then $r=1$).
\end{lem}

\begin{proof}
If $f$ is defined on the whole of $\mysp(\Gamma;V_1)$ and since any element of $\mysp$ is the limit of sums of Dirac functions, it is actually determined by $f(\delta_{e,v})$ where $\delta_{e,v}$ is the Dirac function at the identity with value $v \in V_1$. Let $h \in \mysp(\Gamma,\homo(V_1,V_2))$ be defined by $h(\gamma)(v) = f(\delta_{e,v})(\gamma)$ for $v \in V_1$. Now $\mysp(\Gamma;V) \ni x =  \sum_\gamma \delta_{\gamma,x(\gamma)}$ and $x$ is a limit of the partial sums, hence, by linearity, continuity and $\Gamma$-equivariance of $f$,
\[
f(x)(\eta) = f\big( \sum_\gamma \delta_{\gamma,x(\gamma)} \big)(\eta) = \sum_\gamma \gamma f \big( \delta_{e,v_\gamma} \big)(\eta) = \sum_\gamma \gamma h(\eta) \big( x(\gamma) \big) = \sum_\gamma h(\gamma^{-1} \eta) x(\gamma).
\]
This gives the convolution on the right. %As all possible $x \in \mysp$ are allowed and the above function of $\eta$ is in $\ell^q$, it follows that $f$ must be of $\ell^r$type with $\tfrac{1}{p} + \tfrac{1}{r} \geq \tfrac{1}{q} +1$.
\end{proof}
\par The above argument fails if $\mysp = \ell^\infty$ as its elements are not all limits of the partial sums of Dirac elements. See example \ref{cexinv} for a $\Gamma$-invariant map which cannot be written in this way.

\begin{defi}\label{dtype}
%Let $\mysp$ be $\ell^p$ for $ p \in [1,\infty]$ or $c_0$. 
Let $Y_i \subset \ell^\infty(\Gamma;V_i)$ (where $i=1,2$). 
A $\Gamma$-equivariant map $f: Y_1 \to Y_2$ is said to be of $\mysp$ type if there exists an element $h \in \mysp(\Gamma; \homo(V_1,V_2) )$ such that 
\[
f(y)(\eta) = \somme{\gamma \in \Gamma}{} h(\gamma^{-1} \eta) y(\gamma) = \somme{\gamma \in \Gamma}{} h(\gamma) y(\eta \gamma^{-1})  
\]
If there exists a $h$ with finite support, then $f$ is said of finite type. If $f:Y_1 \to Y_2$ is in the closure (in the operator norm) of the maps of finite type then $f$ will be called a pseudo-function (the closure depends on $Y_1$ and the norm in $Y_2$). 
\end{defi}

The standard terminology is $p$-pseudo-function for a pseudo-function from $\ell^p \to \ell^p$. Perhaps even more common is the name reduced $C*$-algebra (and notation $C^*_r(\Gamma)$) for a pseudo-functions from $\ell^2(\Gamma) \to \ell^2(\Gamma)$. Note that maps of $\ell^1$ type are always pseudo-functions thanks to Young's inequality.

In other words, $f$ is of $\mysp$ type if it is a convolution by an element of $\mysp$. Note all convolutions (of the above form) are $\Gamma$-equivariant maps. %As a side remark, a subspace $Y \subset \mysp(\Gamma;\rr)$ which is $\Gamma$ invariant is closed if and only if convolution by an element of $\ell^1(\Gamma;\rr)$ has image contained in $Y$.

%-------------------------------------------------------------------------
\section{Dual $\ell^p$ dimension.}\label{sdual}
%-------------------------------------------------------------------------

\par Given a subspace $Y$ of a Banach space $X$, its dual is not actually a subspace of $X^*$, the dual of $X$, but a quotient of that space. More precisely, let $Y^\perp = \{ x^* \in X^* \mid \forall y \in Y, \pgen{x^*,y} =0\}$ be the annihilator of $Y$ (see among many possible choices \cite[Section 4.6]{Rud}), then $Y^* = X^*/Y^\perp$. 

\par Consequently, the dual of a linear subspace is quite awkward to work with as in general the norm on the quotient is not so nice. But $Y^\perp$ remains a reasonably tractable space. In particular, if $Y \subset \ell^p(\Gamma;V)$ is a $\Gamma$-invariant linear subspace, then $Y^\perp$ is a closed (actually, weak$^*$-closed) $\Gamma$-invariant linear subspace of $\ell^{p'}(\Gamma;V)$.

\begin{defi}
Let $V$ be a finite-dimensional normed vector space. Let $p\in[1,\infty]$ and $Y \subset \ell^p(\Gamma;V)$ (or $c_0(\Gamma;V)$ if $p=\infty$) be a subspace invariant by the natural action of $\Gamma$, an amenable countable group. Then, the dual $\ell^p$ dimension of $Y$ is defined by
\[
\ddlp Y = \dim_\kk V - \dl{p'}Y^\perp 
\]
\end{defi}

\par In order to avoid some typical problems related to the spaces $\ell^1$ and $\ell^\infty$, a convention will be made in throughout this section. When $p=\infty$, $\ell^p$ should be read as $c_0$. However, when $p=1$, the space $\ell^{p'}$ should still be understood as $\ell^{\infty}$.

\par In order to speak of the properties of $\ddlp$, it might be worth to make a rapid tour of some basic properties of $Y^\perp$. But before, define, for $Z \subset X^*$, the weak annihilator as $Z^\perpud = \{ x \in X \mid \forall x^* \in Z, \pgen{x^*,x} =0 \}$. Let  $Y \subset \ell^p(\Gamma;V)$ and $Y^\perp \subset \ell^{p'}(\Gamma;V)$ be its annihilator, let $Z \subset \ell^{p'}(\Gamma;V)$ and $Z^\perpud \subset \ell^{p}(\Gamma;V)$ be the weak annihilator, then

\begin{enumerate}\renewcommand{\labelenumi}{{\normalfont D\arabic{enumi} - }}
\item $Y = \{0\}  \ssi Y^\perp = \ell^{p'}(\Gamma;V)$ and $Y = \ell^{p}(\Gamma;V) \ssi Y^\perp = \{0\}$. 
\item $Y^\perp = (\overline{Y})^\perp$.
\item $Y_1 \subset Y_2 \imp Y_2^\perp \subset Y_1^\perp$ and $ Z_1 \subset Z_2 \imp Z_2^\perpud \subset Z_1^\perpud$.
\item $(Y_1 \oplus Y_2)^\perp = Y_1^\perp \oplus Y_2^\perp$.
\item $Z^\perpud = \{0\} \imp \srl{Z} = \ell^{p'}(\Gamma;V)$.
\item $(Y^\perp)^\perpud = \srl{Y}$.
\end{enumerate}

\par P1 is obtained quickly as a consequences of these properties (and that $\dlp \{0\} =0$). The same can be said for properties P3-P6. Invariance requires a bit of work.

\begin{prop}
P2 for $\dlp$ implies P2 for $\ddlp$. That is P2 holds for $\ddlp$ (if $p \in ]1,\infty[$) in $\ell^p$ and for $\ddl{\infty}$ in $c_0$.
\end{prop}

\begin{proof}
Let us proceed in three steps. First, if $Y \subset Y'$ then $Y^\perp \supset Y'^\perp$, so $\dl{p'} Y^\perp \geq \dl{p'} Y'^\perp$. This obviously implies that $\ddlp Y^\perp \leq \ddlp Y'^\perp$. So P2 holds for inclusion maps, and it is possible to assume that $f$ is actually bijective.

\par Let $f:Y \to Y'$ be bijective and of finite type. Suppose further that $f(\ell^p(\Gamma;V)) =\ell^p(\Gamma;V')$ ($f$ being of finite type, it is defined on the whole of $\ell^p$). Let $f^*: \ell^{p'}(\Gamma;V') \to \ell^p(\Gamma;V)$ be the adjoint of $f$, defined by $\pgen{f^* (y'^*) , y} = \pgen{ y'^*, f(y)}$. It is also of finite type, and, since $f(\ell^p(\Gamma;V)) = f(\ell^{p}(\Gamma;V'))$, it is injective. Thus $f^*:Y'^\perp \to Y^\perp$ is injective and of finite type so $\dl{p'} Y'^\perp \leq \dl{p'} Y^\perp$.

\par Finally, if $f(\ell^p(\Gamma;V)) \subsetneq \ell^{p}(\Gamma;V')$ then look at $Z = Y \oplus \{0\} \subset \ell^p(\Gamma;V \oplus V')$. A simple extension of $f$ to $\ell^p(\Gamma; V \oplus V')$ makes it surjective as a whole, while $f(Z) = Y'$ and $\ddlp Z = \ddlp Y$. 
\end{proof}

\par The rest of the properties are in some sort of dual relation. 

\begin{prop}
Property P7 (resp. P8', P9', P10') for $\dl{p'}$ implies property P10 (resp. P9', P8', P7) for $\ddlp$. %If $p \in ]1,\infty[$, then P9' for $\dl{p'}$ implies P8' for $\ddlp$.

Consequently, $\ddlp$ has property P8' and P10 for $p \in [2,\infty]$, property P9' for $p \in ]1,\infty[$. %xxx
\end{prop}
\begin{proof}
(from P7 to P10) Suppose that $\ddlp Y = \dim_\kk V$ or, equivalently, that $\dl{p'} Y^\perp = 0$. Then, as P7 is assumed to hold for $\dl{p'}$, $Y^\perp =\{0\}$, which in turns implies that $\srl{Y} = \ell^p(\Gamma;V)$. 

\par (from P10' to P7) Suppose that $\ddlp Y = 0$ or, equivalently, that $\dl{p'} Y^\perp = \dim_\kk V$. Given that P10 holds for $\dl{p'}$, $Y^\perp$ is of finite codimension in $\ell^{p'}(\Gamma,V)$, which in turns implies that $Y$ is finite dimensional (because $Y^* = X^*/Y^\perp$ is finite dimensional). Since $Y$ is $\Gamma$-invariant and $Y \subset c_0(\Gamma,V)$, $Y$ is infinite dimensional, a contradiction. %\marginpar{only need Y w$^*$ closed!}

\par (from P8' to P9') Given an increasing sequence $Y_n$ with $\srl{\cup Y_n} = \ell^p(\Gamma;V)$, then the $Y_n^\perp$ form a (closed) decreasing sequence with $\cap Y_n^\perp = \{0\}$. Indeed, if $0 \neq y^\perp \in \cap Y_n^\perp$ then $\exists x \in \ell^p(\Gamma;V)$ with $\pgen{y^\perp, x} =1$. But there is also a sequence $y_n \in Y_n$ with $y_n \to x$, which would contradict the continuity of $y^\perp$. 

\par (from P9' to P8') Recall $(Y^\perp)^\perpud = \srl{Y}$. Given a decreasing sequence of closed spaces $Y_n$ with $\cap Y_n = \{0\}$, then $Y_n^\perp$ form an increasing sequence. Let $Z = \srl{\cup Y_n^\perp}$. Then $Z^\perpud = (\cup Y_n^\perp)^\perpud = \cap \big( Y_n^{\perp^\perpud} \big) =  \cap \srl{Y_n} = \{0\}$. Consequently, $\srl{Z} = \ell^{p'}(\Gamma;V)$ and, using P9' for $\dl{p'}$ one gets that $\dl{p'} Y_n^\perp \to \dim_\kk V$. By definition, $\ddlp Y_n \to 0$.
\end{proof}

\par This argument also works to show that P9' for $\dl{\infty}$ in $\ell^\infty$ implies P8' for $\ddl{1}$, but since P9' is false for $\dl{\infty}$ in $\ell^\infty$ this is not of great interest. Actually, taking the $Y^{(1)}_n$ as in \cite[Example 4.2]{moi-lp}, it is quite easy to see that P8' fails for $\ddl{1}$ and that they are the weak annihilators of the $Y^{(\infty)}_n$ of example \ref{cexlcont}.

\section{Positivity} \label{spos}
%----------------------------------------------------------------------------

%In this section, positivity is done. Namely, it is not required to go through the (non-constructive) existence of an element in $\ell^1$. 
The problem of positivity is akin to a question studied by Edgar and Rosenblatt in \cite{ER}. In the cited article they show that if $\Gamma$ is a locally compact Abelian group without compact subgroup, then any $0 \neq f \in L^p(G)$, for $p \in [1,2]$, has linearly independent translates.  The question will here again be reduced to the existence, for each $\Omega_i$, of a map $B_i:\kk^{n_i} \to Y$ with $n_i \geq c |\Omega_i|$. Such maps will be realized by looking at a map sending $\{a_i\} \in \kk^n$ to the linear combination $\sum a_i \gamma_i y$ of translates of some element $y$.

\par Since $Y \neq \{0\}$, there exists $y \in Y \subset \ell^p(\Gamma;V)$ which can be renormalized so that $\|y\|_{\ell^p(\Gamma)} =1$. Then, for $\eps_0 \in ]0,1[$, choose a set $F_{\eps_0}$ such that $\|y\|_{\ell^p(F_{\eps_0}^\comp)} \leq \eps_0$. As $y$ is normalized, $\|y\|_{\ell^p(F_{\eps_0})} \geq (1-\eps_0^p)^{1/p}$.
\par Let $R_\Omega$ be the restriction map:
\[
\begin{array}{rccc} 
R_\Omega : & \ell^p(\Gamma;V) & \to     & \ell^p(\Omega;V) \\
           &  f               & \mapsto & f \bv_\Omega
\end{array}
\]
Let $y^* \in \ell^{p'}(\Gamma;V^*)$ be such that it is  supported on $F_{\eps_0}$, $\pgen{ y^* , R_{F_{\eps_0}} y} =1$ and $\|y^*\|_{\ell^{p'}(\Gamma)} \leq (1-\eps_0^p)^{-1/p}$. Let $G_i$ be the subset (depending on $\Omega_i$, but also on $F_{\eps_0}$) obtained from lemma \ref{tcardtil}. 

To construct the ``thick'' linear subspace, introduce
\[
\begin{array}{cccc}
I:& \kk^{G_i} &\to     & \ell^p(\Gamma;V) \\       
  & \{a_\gamma\}_{\gamma \in G_i} &\mapsto & \somme{\gamma \in G_i}{} a_\gamma (\gamma y) 
\end{array}
\textrm{ and }
\begin{array}{cccc}
\pi: & \ell^p(\Gamma;V) &\to     & \kk^{G_i} \\       
     & y                &\mapsto & \{  (\gamma y^*) y \}_{\gamma \in N_i}. 
\end{array}
\] 
The map $I$ puts a $\kk^{G_i}$ in $Y$ by sending elements of the (usual) basis to translates of $y$ that are sufficiently far away from each other in the hope that their linear combinations will not interfere too much. To avoid confusion due to notations, please note that $\ell^p(|G_i|)$ is $\kk^{G_i}$ with the $\ell^p$ norm and is not to be confused with $\ell^p(G_i;V)$.

\begin{lem}\label{tpertub}
Let $Q = \pi \circ I$ be as above, let $\eps_1 := \eps_0 / (1-\eps_0^p)^{1/p}$ and let $\{e_k\}$ the usual basis of $\kk^{G_i}$, then
\[
\| Q e_k - e_k \|_{\ell^p(|G_i|)} \leq \eps_1 
\]
\end{lem} 

\begin{proof}
It suffices to make a direct calculation of $ Q e_k - e_k = \sum_{j \neq k} \langle y_j^*, y_k \rangle e_j$. Indeed,
\[
\begin{array}{rcl}
\| Q e_k - e_k \|_{\ell^p(N_i)}^p 
  & \leq & \sum_{j \neq k} |y^*_j y_k |^p \\
  & \leq & \sum_{j \neq k} \|y^*_j\|_{\ell^{p'}}^p \|y_k \|_{\ell^p(\gamma_j F_{\eps_0}) }^p \\
  & \leq & (1-\eps_0^p)^{-1} \|y \|_{\ell^p(F_{\eps_0}^\comp)}^p \\
  & \leq & (1-\eps_0^p)^{-1} \eps_0^p,
\end{array}
\]
which, upon setting $\eps_1 = \eps_0 / (1-\eps_0^p)^{1/p}$, is the claim of the lemma.
\end{proof}

\begin{lem} \label{tborndim=>pos}
Let $Q$ be as above. Suppose there exists $\delta_0 \in [0,1]$ such that for every $ \delta < \delta_0$ there exists $c(\delta) \in [0,1[$ with the following property: every subspace $X^k \subset \kk^{G_i}$ of dimension $k>0$ such that $ \forall v \in X^k, \| Q v\|_{\ell^p} \leq \delta \| v \|_{\ell^p}$, satisfies $k \leq c(\delta) |G_i|$. Then 
\[
\dlp Y \geq \bigg( 1- c(0) - \inff{\delta \in [0,\delta_0] } c(\delta) \bigg) \limm{i \to \infty} \frac{\beta(\Omega_i;F_{\eps_0})}{|\Omega_i|} \geq \big( 1-c(0)-\limm{\delta \to 0} c(\delta) \big) \frac{1}{|F|^2}
\]
\end{lem}

\begin{proof}
From definition \ref{defdlp}, what must be produced is a lower bound of $\ldm_\eps (B^{Y,p}_1, \ev_{\ell^p(\Omega_i)})$. So suppose there is a linear subspace $X' \subset \ell^p(\Gamma;V)$ such that $\diam_{\ev_{\ell^p(\Omega_i)}} X' \cap B^{Y,p}_1 \leq \eps$ and $\dim X' > |\Omega_i| \dim V - \big( 1- c(0) - c(\delta) \big) |G_i|$. 

\par Then, $X' \cap \img I > c(\delta) |G_i|$, as $R_{\Omega_i} \circ I:\kk^{G_i} \to \ell^p(\Gamma;V)$ has kernel of dimension at most $c(0)|G_i|$. Consequently there exists a subspace $X \subset \kk^{G_i}$ of dimension greater than $c(\delta) |G_i|$ with $I(X) \subset X'$. Thus 
\[
\forall v \in X, \| Q v \|_{\ell^p} \leq \|\pi\|_{\ell^p(\Gamma;V) \to \kk^{G_i}} \|Iv\|_{\ell^p} \leq \|y^*\|_{\ell^\infty} \eps.
\]
\par If $\|y^*\|_{\ell^\infty} \eps < \delta$, this gives a contradiction. As a result, $\ldm_\eps (B^{Y,p}_1, \ev_{\ell^p(\Omega_i)}) \geq \big(1-c(0)-c(\delta) \big) |G_i|$ for $\eps < \delta \tfrac{ \eps_0}{(1-\eps_0^p)^{1/p}}$. Since this estimate can be redone for arbitrarily small $\eps$ (and consequently $\delta$), lemma \ref{tcardtil} gives the conclusion.
\end{proof}

The proofs of positivity for $p=1$ and $2$ can be summarized as follows. First, given a $N \times N$ matrix $M$ and any $p \in [1,\infty]$, define the two following norms
\[
\begin{array}{lcl}
\| M \|_{\ell^p(N^2)}         & = & \big( \sum |m_{ij}|^p \big)^{1/p}, \\
%\| M \|_{\ell^{p'}(r);\ell^q} & = & \Big( \sum_i \big( \sum_j |m_{ij} |^{p'} \big)^{q/p'} \Big)^{1/q}, XXX \\
\| M \|_{\ell^p \to \ell^q}   & = & \supp{ \| v \|_{\ell^p(N)} = 1} \|M v\|_{\ell^q(N)}.
\end{array}
\] 
%It is easily seen that $\|M\|_{\ell^p \to \ell^q} \leq \| M \|_{\ell^{p'}(r) ;\ell^q}$.

\par It must be shown that if there is a subspace $X$ of dimension $k>0$ such that $ \forall v \in X, \| Q v\|_{\ell^p} \leq \delta \| v \|_{\ell^p}$, then $k \leq c N$, for some constant $ 0 \leq c <1 $ which depends only on $\delta$ and properties of $Q$ (such as $\eps_1$ in \ref{tpertub}).
\par For $p=1$ the argument of % theorem \ref{tpos-facile} (or, equivalently, of 
\cite[Proposition 4.1]{moi-lp}) is basically to construct two opposite bounds on a norm of $M = Q - \Id$. On one hand, if $v \in X$ is as in lemma \ref{tborndim=>pos} (\ie $\|Qv\| \leq \delta \|v\|$) implies that $\| M \|_{\ell^1 \to \ell^1} \geq 1-\delta$. On the other hand,
\[
\| M v \|_{\ell^1(N)} = \| M \sum_j v_j e_j \|_{\ell^1(N)}   \leq \sum_j |v_j| \|M e_j \|_{\ell^1(N)} \leq \sum_j |v_j| \eps = \eps \|v_j\|_{\ell^1(N)},
\]   
which means that $\| M \|_{\ell^1 \to \ell^1} \leq \eps$. So $c$ can be taken to be $0$ as long as $\eps + \delta < 1$.

\par The argument for $p=2$ is actually an estimate on $\| M\|_{\ell^2(N^2)}$. The upper bound is a simple consequence of the norm bound as stated in lemma \ref{tpertub}: $\| M \|_{\ell^2(N^2)}^2 \leq N \eps^2$. The lower bound is obtained by remarking that $\| M \|_{\ell^2(N^2)}^2 = \tr M^t M$ is actually invariant under changes of orthonormal basis. If an orthonormal basis of $X$ is chosen (and completed in an orthonormal basis of $\kk^N$) then the $\ell^2$ norm of the first $k$ columns is bigger than $k^{1/2} (1-\delta)$, and consequently  $\| M \|_{\ell^2(N^2)}^2 \geq k (1-\delta)^2$. Thus $c \leq \big( \frac{\eps}{1-\delta} \big)^2$.

\par The next lemmas are amusing exercises forming the basis of positivity. 

\begin{lem}\label{ttrace}
Let $\{e'_i\}_{1 \leq i \leq N}$ be an orthogonal basis of $\kk^N$, and $M : \kk^N \to \kk^N$ be a linear map such that $M e'_i = e'_i$ for $1 \leq i \leq k$ then $k \leq \| M \|_{\ell^2(N^2)}^2$.
\end{lem}

\begin{proof}
The $\ell^2(N^2)$ norm for these matrices can also be expressed by $\tr M^t M$ and is consequently independent of the choice of orthogonal basis.
As $Me_i' = e'_i$ for $1 \leq i \leq k$, a simple computation yields $k \leq \tr M^t M = \|M\|_{\ell^2(N^2)}^2$.% \leq \eps^2 N$.  
\end{proof}

\begin{lem}\label{tbornker}
Let $p \in [1,2]$ and let $L : \kk^N \to \kk^N$ be a linear map such that $\| L e_i - e_i \|_{\ell^p(N)} \leq \eps$ for $\{e_i\}_{1 \leq i \leq N}$ the usual basis of $\kk^N$. Then $\dim \ker L \leq \eps^2 N$. 
\end{lem}

\begin{proof}
Let $\dim \ker L = k$ and $M = L -\Id$. Since there is an orthogonal (in the usual $\ell^2$ sense!) basis of $\kk^N$ such that the first $k$ elements actually form a basis of $\ker L$, lemma \ref{ttrace} implies $\|M\|_{\ell^2(N^2)}^2 \geq k$. 
\par On the other hand, the hypothesis means that $\|M e_i\|_{\ell^p(N)} \leq \eps$. However, for $p \in [1,2]$, $\|M e_i\|_{\ell^2(N)} \leq \|M e_i\|_{\ell^p(N)}$ thus $\|M\|_{\ell^2(N^2)}^2  = \sum_i \|M e_i\|_{\ell^2(N)}^2 \leq N \eps^2$. It follows that  $k \leq \eps^2 N$ which is the claim of the lemma.
\end{proof}

%On the other hand, the hypothesis means that $\|M\|_{\ell^p(N^2)}^p \leq N \eps^p $ (for the $\ell^p$-norms of $N \times N$ matrices). 
%However, for $p \in [1,2]$, $ \| M \|_{\ell^2(N^2)} \leq \| M \|_{\ell^p(N^2)}$ so
%\[
%k \leq \|M\|_{\ell^2(N^2)}^2 \leq \| M \|_{\ell^p(N^2)}^2 \leq N^{2/p} \eps^2.
%\]
%However, for $p \in [2,\infty]$, $ \| M \|_{\ell^2(N^2)} \leq N^{1 - 2/p} \| M \|_{\ell^p(N^2)}$ so
%\[
%k \leq \|M\|_{\ell^2(N^2)}^2 \leq N^{2-4/p} \| M \|_{\ell^p(N^2)}^2 \leq N^{2-2/p} \eps^2.
%\]
%\begin{lem}\label{tbornkerl2}
%Let $L : \rr^N \to \rr^N$ a linear map so that $\| L e_i - e_i \|_{\ell^2(N)} \leq \eps$ for $\{e_i\}_{1 \leq i \leq N}$ the usual basis of $\rr^N$. Then $\dim \ker L \leq \eps^2 N$. 
%\end{lem}
%\begin{proof}
%Let $\dim \ker L = k$ and $M = L -\Id$. The hypothesis means that $\|M\|_{\ell^2(N^2)}^2 \leq N \eps^2 $ (for the $\ell^2$-norms of $N \times N$ matrices). This norm can actually be rewritten as $\tr M^t M$ and is consequently invariant under changes of basis. Let $\{e'_i\}_{1\leq i \leq N}$ be another basis such that $\{e_i'\}_{1 \leq i \leq k}$ is a basis of $\ker L$. As $Me_i' = e_i$ for $1 \leq i \leq k$, a simple computation yields $k \leq \tr M^t M = \|M\|_{\ell^2(N^2)}^2 \leq \eps^2 N$.  
%\end{proof}

\begin{teo}\label{tpos}
Let $p \in [1,2]$ and $Y \subset \ell^p(\Gamma;V)$. If $Y \neq \{0\}$, and there is a $y \in Y$ satisfying $\| y \|_{\ell^p(F^\comp)} \leq \eps_0$ for some finite subset $F \subset \Gamma$, then $\dlp Y \geq (1-2\eps_1^2) |F|^{-2}$ where $\eps_1 = \tfrac{\eps_0}{(1-\eps_0^p)^{1/p}}$  
\end{teo}

\begin{proof}
\par The aim is to show that lemma \ref{tborndim=>pos} applies. Let $N = |G_i|$. Let $X \subset \kk^{G_i}$ be such that $\dim X =k$ and $\forall v \in X$, $\|L v\|_{\ell^p(N)} \leq \delta \| v \|_{\ell^p(N)}$. Let $f_j$ be an orthonormal basis for $X$ which minimizes $C_{X;p} = \tfrac{1}{k} \sum_j \|f_j\|_{\ell^p(N)}^2$. Obviously, $C_{X;p} \geq 1$, but more importantly, $C_{X;p} \leq 4 (\tfrac{N}{k})^{\frac{2}{p}-1} + 4$. For $r \in [0,1]$, let 
\[
C_{r;p} = \sup \{ C_{X';p} \mid X' \subset \kk^N, \dim X' = k \geq N r \}
\] 

\par As $p \in [1,2]$, $\| L f_j \|_{\ell^2(N)} \leq \| L f_j \|_{\ell^p(N)} \leq \delta \| f_j \|_{\ell^p(N)}$. Replace $L$ by a matrix $L'$ such that $X \subset \ker L'$; the estimate $\| L -L' \|_{\ell^2(N^2)}^2 \leq k \delta^2 C_{X;p}$ follows from the last equation and the invariance of the $\ell^2(N^2)$ norm under change of orthonormal basis.

\par Let $M = L - \Id$ and $M' = L' - \Id$. Then $\|M - M'\|_{\ell^2(N^2)} = \| L - L' \|_{\ell^2(N^2)} \leq k^{1/2} \delta C_{k/N;p}^{1/2}.$

Since $\| M' \|_{\ell^2(N^2)}^2 \geq k$ by lemma \ref{ttrace}, $\| M \|_{\ell^2(N^2)} \geq k^{1/2} - k^{1/2} \delta C_{k/N;p}^{1/2}$. Together with $\| M \|_{\ell^2(N^2)} \leq N^{1/2} \eps_1$ where $\eps_1 = \tfrac{\eps_0}{(1-\eps_0^p)^{1/p}}$, this yields
\[
\big( \tfrac{k}{N} \big)^{1/2} ( 1-  \delta C_{k/N;p}^{1/2}) \leq \eps_1.
\]
So if $\frac{k}{N} > c(\delta) = \max( c_1(\delta) , c_2(\delta) )$, where $c_1(\delta) = \inf \{ r \in [0,1] \mid C_{r;p} < \delta^{-1} \} < \frac{4\delta}{1-4\delta}$ and $c_2(\delta) = \frac{\eps_1}{1-\delta^{1/2}}$, then there is a contradiction.

\par Using this $c(\delta)$ in lemma \ref{tborndim=>pos}, gives
\[
\dlp Y \geq (1- 2 \eps_1^2) \limm{i \to \infty} \frac{\beta(\Omega_i;F)}{|\Omega_i|} \geq \frac{(1- 2 \eps_1^2)}{|F|^2}. \qedhere 
\]
%Now suppose $k = f(N_i) N_i$ for some function $f: \nn \to [0,1]$ such that  $\limm{N_i \to \infty} f(N_i) = 1$. Then, since $C_{k,N_i;p}  \to 1$, $1 -\delta \leq \eps $ which is a contradiction.
\end{proof}
Let us indulge in a few remarks about the above estimate. First, the estimate is trivial if $\eps_0 \geq (2^{p/2}+1)^{-1/p}$. Second, for $p=1$, it is less useful than 
% theorem \ref{tpos-facile} (or equivalently 
\cite[Proposition 4.1]{moi-lp} which gives the bound $\dl{1} Y  \geq |F|^{-2}$ as long as $\eps_0 < 1/2$. This is not only better in the value of the dimension obtained but in the interval of admissible $\eps_0$ (as $1/2 > (2^{1/2}+1)^{-1}$). Third, suppose $|F|=1$ and $V=\kk$. Then, since the von Neumann dimension can also be defined as the projection of the Dirac mass at the unit $e_\Gamma$ then evaluated at $e_\Gamma$, one has a simple bound $\dl{2} Y \geq (1-\eps_0^2)^{1/2} $. Whereas, in the same conditions, the bound above is $\dl{2} Y \geq 1 - \tfrac{2\eps_0^2}{1-\eps_0^2}$. A Taylor expansion should easily convince the reader that the result of theorem \ref{tpos} is not optimal for $p=2$ either.

%-----------------------------------------------------------------------------
\section{Continuity, completion and maximality} \label{scont}
%-----------------------------------------------------------------------------

There are very weak types of continuity which hold without any problem, as a simple consequence of inclusions:

\par P8'' (right semi-continuity) Let $Y_n$ be an decreasing sequence of closed subspaces and let $Y = \cap Y_n$ then $\limm{n \to \infty} \dlp Y_n \geq \dlp Y$.
\par P9'' (left semi-continuity) Let $Y_n$ be an increasing sequence of subspaces and let $Y = \overline{\cup Y_n}$ then $\limm{n \to \infty} \dlp Y_n \leq \dlp Y$.

\par However, asking for a bit more is not, in full generality, possible. As  \cite[Example 4.2]{moi-lp} shows, $\dl{1}$ is not right continuous at $\{0\}$ (\ie P8' is false for $\dl{1}$). As a reminder, the pathological sequence of spaces is given by the following construction. Let $\pi_n : \ell^1(\zz;\kk) \to \ell^\infty( \zz / n \zz ; \kk)$ be defined, for $y \in \ell^1(\zz;\kk)$ and $0 \leq k < n$, by $\pi_n(y) (k) = \sum_i y(k+ni)$. Then $Y^{(1)}_n = \ker \pi_n$ satisfies $\forall n, \dl{1} Y^{(1)}_n = 1$ and $\cap_n Y^{(1)}_n = \{0\}$. Let us give another (though related) example where some type of continuity fail.

\begin{ex}\label{cexlcont}
Let $Y^{(\infty)}_n \subset \ell^\infty(\zz; \kk)$ be the linear subspace of sequence of period $n$. Being a space of dimension $n$ it is clear that $\forall n \in \nn,  \dl{\infty} Y^{(\infty)}_n =0$. 

\par Let $\yper = \cup Y^{(\infty)}_n$ be the linear subspace of all periodic sequences. Then on any interval $[-k,k] \subset \zz$, $\yper$ is not distinguishable from $\ell^\infty$. In particular, $\forall \eps \in ]0,1[, \ldm_\eps (B^{\yper}_1, \ev_{\ell^\infty([-k,k])}) = 2k +1$. As a consequence, $\dl{\infty} Y = 1$.

\par This is a counterexample to P10, as $\yper \subset \yperf$, so that $\yperf$ is closed and has the full dimension but is clearly not of finite codimension. It is also a counterexample to P9' (and P9) as an increasing sequence (take $Y'_k = \cup_{1 \leq i \leq k} Y^{(\infty)}_i$)  of null-dimensional space has a full dimensional one as a limit.
\end{ex}

\par As in \cite[Example 4.2]{moi-lp}, this phenomenon is due to the quite peculiar nature of $\ell^\infty$ and it seems plausible that it does not happen in the other spaces under consideration (not even in $c_0$). 
These two examples are, in some sense, dual: the above $Y^{(\infty)}_n$ are contained in the annihilators of the $Y^{(1)}_n$ of \cite[Example 4.2]{moi-lp}.

\begin{prop}\label{tcontlp}
Let $p \in [1,2]$ and let $Y_n$ be an increasing sequence of $\Gamma$-invariant subspaces of $\ell^p(\Gamma;V)$. If $\srl{\cup Y_n} = \ell^p(\Gamma;V)$, then $\limm{n \to \infty} \dlp Y_n = \dim V$. In other words, P9' is true for $\dlp$ in $\ell^p$ when $p \in [1,2]$.
\end{prop}

\begin{proof}
Assume first that $V = \kk$. Since $\srl{\cup Y_n} = \ell^p(\Gamma;V)$, there is a sequence $y_n \in Y_n$ such that $y_n \to \delta_e$. %By lemma \ref{contl1}, this sequence can be taken to be in $\ell^1$. 
Using the positivity estimates of theorem \ref{tpos} and that $\beta(\Omega_i;\{e\}) = |\Omega_i|$, one gets that $\forall \eps, \exists N_\eps$ such that $\dlp Y_{N_\eps} \geq 1 - 2 \eps^2 /(1-\eps^p)^{2/p}$.
Since $\beta(\Omega_i;\{e\}) = |\Omega_i|$, one gets that $\forall \eps, \exists N_\eps$ such that $\dlp Y_{N_\eps} \geq 1 - 2 \eps^2 /(1-\eps^p)^{2/p}$.

\par If $V \neq \kk$ then consider $\zz_d$ the cyclic group of order $d = \dim_\kk V$, and let $\Gamma' = \Gamma \times \zz_d$. Let $Y_n'$ be the spaces $Y_n$ seen as in $\mysp(\Gamma';\kk)$. Although they are not $\Gamma'$-equivariant, there are, for $n$ big enough, $d$ functions whose mass $> 1- \eps_0$ at each element of $(e_{\Gamma, i})$ (here $1 \leq i \leq d$). The arguments of theorem \ref{tpos} (and its previous lemmas) apply to these spaces to give the desired estimate.   
\end{proof}

The next proposition is a small strengthening of P3.

\begin{prop}\label{tcomp}
Let $p \in [1,\infty]$, let $Y \subset Y' \subset \ell^p(\Gamma;V)$ be two $\Gamma$-invariant subspaces. If $B^{Y}_1$ is $\tau^*$-dense in $B^{Y'}_r$ (for some $r \in ]0,1]$), then $\dlp Y =\dlp Y'$. In particular,  $\dlp Y = \dlp \srl{Y}$.
\end{prop}

\begin{proof}
Roughly speaking, when restricted to a finite $\Omega \subset \Gamma$, these two spaces cannot be distinguished. More precisely, there exists a linear map, given by the restriction $R_\Omega$, and whose kernel is contained in the ``ball'' of radius $0$:
\[
R_\Omega: (B^{Y'}_r, \ev_{\ell^p(\Omega)})  \to  (R_\Omega B^{Y}_1, \ev_{\ell^p(\Omega)}).
\]
Thus, $\forall \eps \in [0,1]$, $\ldm_{\eps/r} (B^{Y'}_1, \ev_{\ell^p(\Omega)}) = \ldm_\eps(B^{Y'}_r, \ev_{\ell^p(\Omega)}) \leq \ldm_\eps(R_\Omega B^{Y}_1, \ev_{\ell^p(\Omega)})$.
\par On the other hand, let $s: R_\Omega B^{Y}_1 \to B^{Y}_1$ such that $R_\Omega \circ s = \Id$ be determined by an inverse of $R_\Omega Y \to Y$, then $s$ is a linear map which increases distances. Consequently, $\ldm_\eps(R_\Omega B^{Y}_1, \ev_{\ell^p(\Omega)}) \leq \ldm_\eps(B^{Y}_1, \ev_{\ell^p(\Omega)})$. Finally, by inclusion $Y \subset \srl{Y}$, we have \linebreak $\ldm_\eps(B^{Y}_1, \ev_{\ell^p(\Omega)}) \leq  \ldm_\eps(B^{\srl{Y}^{w^*}}_1, \ev_{\ell^p(\Omega)})$.
\end{proof}

A consequence of this remark, is if $B^Y_1$ is weak$^*$ dense in $B^{\ell^p}_r$ then $\dlp Y = \dim_\kk V$. Note that this is the case of all the spaces of full dimension exhibited ($Y^{(1)}_n$ in $\ell^1$, see \cite[Example 4.2]{moi-lp}; $\Pi$ (see example \ref{cexlcont}) or $c_0$ in $\ell^\infty$).

Let us now turn to maximality. The case $p \in [2,\infty]$ (or $c_0$) is the simplest. It is also convenient to begin with $V = \kk$.
\begin{prop}\label{maxlp+2}
Let $p \in [2,\infty[$. If $Y \subset \ell^p(\Gamma)$ (resp. $c_0(\Gamma)$) is a closed $\Gamma$-invariant space such that $\dlp Y = 1$ (resp $\dl{\infty} Y = 1$) then $Y = \ell^p(\Gamma)$ (resp. $c_0(\Gamma)$).
\end{prop}
\begin{proof}
For any $\alpha \in ]0,1[$, there is an $\eps>0$ such that for any $\Omega$ large enough in a F{\o}lner sequence, there is a linear space $L$ of size at least $\alpha |\Omega|$ in $\ell^p(\Omega)$, such that if $y \in L$ and $\|y\| = 1$, then there is a $y' \in Y$ with $\|y'\|_{\ell^p(\Gamma)} \leq \eps^{-1}$ and $y$ and $y'$ agree on $\Omega$. Let $x_\gamma$ a point in $L$ closest to the Dirac mass $\delta_\gamma$. Note this point is actually unique (by strict convexity) unless the ambient space is $c_0$ (see \cite[Section 2.2, page 40]{BL}). This point is closer than the point that one would find assuming things happen in $\ell^2(\Omega)$:
\[
\sum \|\delta_\gamma - x_\gamma\|^2_{\ell^p(\Omega)} \leq  \sum \|P_{L^\perp} \delta_\gamma \|^2_{\ell^p(\Omega)} \leq \sum \|P_{L^\perp} \delta_{\gamma} \|^2_{\ell^2(\Omega)} = \sum \pgen{ P_{L^\perp} e_{\gamma}, P_{L^\perp} e_{\gamma} } = \tr P_{L^\perp} \leq (1-\alpha) |\Omega| 
\]
In particular, the average squared distance between a Dirac mass and $L$ is less than $1-\alpha$. This means at least half of the Dirac mass at are distance less than $\sqrt{2(1-\alpha)}$ from $L$ and this closest element is of norm (in $\ell^p(\Gamma)$) at most $\eps^{-1}$. Take an increasing sequence $y_{\Omega_n}$ of such elements, where $\Omega_n$ is a F{\o}lner sequence. Being bounded and since functions of finite support are dense in the dual (this excludes $\ell^\infty$ but not $c_0$), this sequences converges weakly to some element $y$ which at distance less than $\sqrt{2(1-\alpha)}$. However, weak convergence implies norm convergence of some convex combination (a consequence of the Hahn-Banach theorem, see \cite[Theorem 3.13]{Rud}). Thus, since $Y$ is closed, there is an element in $Y$ which is at distance $\sqrt{2(1-\alpha)}$ from the Dirac mass. Letting $\alpha \to 1$ gives a sequence which tends to the Dirac mass. Since the closure of the Dirac mass and its translates is $\ell^p$, this implies $Y = \ell^p$.
\end{proof}
For the case where $\dim_\kk V >1$, a problem seems to come up: half of the Dirac mass might lay only in the half of the space $V$. This is easily remedied. If an average (of positive numbers) is less than $c$ then a fraction $\delta$ of them is less than $c/\delta$. So it suffices to pick $\delta = \dim_\kk V /(\dim_\kk V +1)$. It will worsen the approximation by a (multiplicative) constant which does not matter so much.

% A small lemma is required before looking at the case $p \in ]1,2]$. For $p\in ]1,\infty[$, let $\maz: \ell^p \to \ell^{p'}$ be the Mazur map, defined by $\maz(f) (\gamma) = |f(\gamma)|^{p-2} f(\gamma)$.

The case $p \in ]1,2]$ is still unclear. The following lemma might be of help. For $p\in ]1,\infty[$, let $\maz: \ell^p \to \ell^{p'}$ be the Mazur map, defined by $\maz(f) (\gamma) = |f(\gamma)|^{p-2} f(\gamma)$.

\begin{lem}\label{tracelp}
Let $p \in ]1,\infty[$. Assume $L \subset \ell^p(\nn)$ is a closed linear space and let $\delta$ some Dirac mass. Let $x$ be the closest point to $\delta_e$ in $B^L_1$. Then
\[
\pgen{ \maz( \delta - x) , x } =0 \quad \textrm{ and } \quad \|x\|^p_{\ell^p} = \pgen{ \maz( \delta - x) + \maz(x) , x }.
\]
\end{lem}
When $p=2$, this is trivial. If $P_L$ is the projection on $L$, then its adjoint $P_L^*$ is the inclusion $L \inj \ell^2(\nn)$. Thus $\| x \|^2_{\ell^2} = \pgen{ P_L \delta, P_L \delta} = \pgen{ \delta, P_L \delta}$. This is consistent with the above formula since, if $p=2$, $\maz$ is the identity.
\begin{proof}
Assume $x$ is as above. Then for any $x' \in L$, the function $f(t) = \| \delta_e -x + t x'\|^p$ has a minimum at $t=0$. This implies that $\maz(\delta-x) \in L^\perp$. In particular, $\pgen{ \maz(\delta-x) , x} =0$ and the conclusion follows upon remarking that $\pgen{\maz (x), x} = \|x\|^p_{\ell^p}$.  
\end{proof}
\section{Additivity and Reciprocity} \label{sadd}
%-----------------------------------------------------------------------------

Given $\Gamma$-invariant subspaces $Y_j \subset \ell^p(\Gamma;V_j)$ where $2 \leq j \leq k$ and the $V_j$ are finite dimensional normed (and without loss of generality Hilbertian) spaces, there is a natural construction of $Y = \oplus_j Y_j \subset \ell^p(\Gamma; \oplus V_j)$.

\par Sub-additivity of $\dlp$, \ie $Y = Y_1 \oplus Y_2 \imp \dlp Y \leq \dlp Y_1 + \dlp Y_2$ is a direct consequence of sub-additivity of $\ldm$: 
\[
%\ldm_\eps B^{Y,p}_1 \leq \ldm_{2^{-1/p}\eps} B^{Y_1,p}_1 + \ldm_{2^{-1/p}\eps} B^{Y_2,p}_1.
\ldm_\eps X \leq \ldm_{2^{-1/p}\eps} X_1 + \ldm_{2^{-1/p}\eps} X_2,
\] 
for $X \subset X_1 \times X_2$ convex centrally symmetric sets in finite-dimensional $\ell^p$ spaces. This can quite readily be seen: if $L_j^{-k_j}$ is the vector space realizing the width of $X_j$ then $L_1^{-k_1} \oplus L_2^{-k_2}$ gives an upper bound for the width of $X$.
 
\par A good starting point is to study, under the pretence of amnesia, the $\ell^2$ case (as additivity holds for $\dl{2}$ since it equal to the von Neumann dimension, see \cite[Corollary A.2]{moi-lp}).

\par The upcoming lemma is only of use when one actually restricts to a \emph{finite} subset $\Omega \subset \Gamma$, though the reader can extend it to close subspaces when $\Omega$ is infinite. 

\begin{lem} \label{taddl2}
Let $Y = Y_1 \oplus Y_2$ be subspaces, $Y_j \subset \ell^2(\Gamma;V_j)$. Then $\ldm_\eps (B^{Y}_1 , \ev_{\ell^2(\Omega)}) \geq \ldm_{2\eps} (B^{Y_1}_1 , \ev_{\ell^2(\Omega)}) + \ldm_{2\eps} (B^{Y_2}_1 , \ev_{\ell^2(\Omega)})$.
\end{lem}
\begin{proof}
Abbreviate $X := R_\Omega B^{Y,2}_1$, where $R_\Omega: \ell^2 (\Gamma;V) \to \ell^2(\Omega;V)$ is the restriction map, and let us consider everything as inside the subspace $\ell^2(\Omega;V)$.  Indeed, the $\ldm$ of $B^{Y,2}_1$ with the pseudo-norm $\ev_{\ell^2(\Omega)}$ turns out that this ball is ``isometric'' (meaning fibres of diameter $0$) to $R_\Omega B^{Y,2}_1$ with the $\ell^2$ norm on $\ell^2(\Omega;V)$. For simplicity, assume that $V$ is Hilbertian (being of finite dimension this does not account for much distortion).

\par Suppose $\ldm_\eps X =k$, then let $L^{-k}$ be a subspace of codimension $k$ realizing this width, \ie such that $\diam L^{-k} \cap X \leq \eps$. The aim is to construct two subspaces $L_j$ whose intersection with $X_j := B^{Y_j,2}_1$ is small. Let $\pi_L : Y \to L$ be the projection on $L$, $\pi_j : Y \to Y_j$ the projections on the $Y_j$ and $i_j: Y_j \to Y$ the natural inclusion.
\par Each operator $\pi_L \circ i_j \circ \pi_j : L \to L$ is self adjoint and gives rises to an eigenvector basis $\{v_{j,\lambda_m}\}$ of $L$. Actually, 
\[
\pi_L \circ i_2 \circ \pi_2 v_{1,\lambda} = \pi_L (\Id_L - i_1 \circ \pi_1) v_{1,\lambda} = (1-\lambda) v_{1,\lambda},
\]
so that the decomposition is the same for $j=1,2$ up to the involution $\lambda \to 1-\lambda$. Let $L_j' = \gen{v_{j,\lambda_m} | \lambda_m \geq 1/2}$ and $L_j = \pi_j L_j'$. Obviously $\dim L_1' + \dim L_2' \geq \dim L^{-k}$ and $\dim L_j = \dim L_j'$ since $\pi_j$ is injective on $L_j'$ ($\ker \pi_L \circ i_j \circ \pi_j \cap L_j' =0$). It remains only to be shown that each $L_j \cap X_j$ has small diameter.
\par Suppose $\exists x,x' \in L_j \cap X_j$ such that $\|x-x'\|_{\ell^2} \geq \delta$. Let $l, l' \in L_j'$ such that $\pi_j l = x$ and $\pi_j l' = x'$. Then $\pi_L \circ i_j \circ \pi_j l, \pi_L \circ i_j \circ \pi_j l' \in L \cap X$, and consequently, writing $l-l' = \sum_{\lambda_m \geq 1/2} c_m v_{j,\lambda_m}$, 
\[
\eps^2 
    \geq \| \pi_L \circ i_j \circ \pi_j (l-l') \|^2 
    =    \somme{\lambda_m \geq 1/2}{} c_m^2 \lambda_m^2
    \geq \tfrac{1}{4} \somme{\lambda_m \geq 1/2}{} c_m^2
    =    \tfrac{1}{4} \| l-l'\|^2
    \geq \tfrac{1}{4} \| \pi_j (l-l') \|^2
    \geq \tfrac{\delta^2}{4}.
\] 
Thus $\diam L_j \cap X_j \leq 2 \eps$, as desired.
\end{proof}

\par An interesting comment due to B.~Hayes is that a result described in G.~Pisier's book \cite[Theorem 6.1]{Pis} can help getting quite uniform distortion bound for finite dimensional subspaces, even inside an infinite dimensional one. This might simplify the proof in the case $p \in [1,2]$. Inside finite-dimensional spaces, there is a simple distortion factor between the metrics. Thus lemma \ref{taddl2} will, with greater losses as the size is big, pass to the $\ell^p$ setting.

\begin{lem}\label{taddlp}
Let $Y_1$ and $Y_2$ be finite dimensional $\ell^p$ spaces, let $Y = Y_1 \oplus Y_2$, let $n = |\Omega| \dim_\kk V <\infty$ and let $\Delta(p) = |\frac{1}{2} - \frac{1}{p}|$. Then $\ldm_\eps (B^{Y,p}_1,\ev_{\ell^p(\Omega)}) \geq \ldm_{2 n^{\Delta(p)} \eps} ( B^{Y_1,p}_1 , \ev_{\ell^p(\Omega)}) + \ldm_{2 n^{\Delta(p)} \eps} (B^{Y_2,p}_1 ,\ev_{\ell^p(\Omega)})$.
\end{lem}

\begin{proof}
The argument is identical as that of \ref{taddl2} except for the norms estimates. Indeed, $\| \pi_L (x-x') \|_p \leq \eps$ implies only that $\| \pi_L (x-x') \|_2 \leq (\un_{p \in [1,2]} + \un_{p \in ]2,\infty]} n^{ 1/2 - 1/p }) \eps$. Similarly from the assumption $\| x-x'\|_p \geq \delta$ no more than $\| x - x'\|_2 \geq (\un_{p \in [1,2[} n_j^{ 1/p - 1/2 } + \un_{p \in [2,\infty]}) \delta$ can be deduced, where $n_j = |\Omega| \dim_\kk V_j \leq n$. The corresponding conclusion becomes $\delta \leq 2 n^{\Delta(p)} \eps$.
\end{proof}

\par The upcoming proof is very close the proof of the independence of $\dlp$ on the choice of F{\o}ner sequence. An important ingredient from the latter proof (which turns out to be a corollary to some extension of the Ornstein-Weiss lemma, \cite[Theorem 5.1]{moi-lp}) is Helly's selection principle which allows the limit to be computed from a particular subsequence. To alleviate the notations, the following shorthands will now be in use: $a^Y(\Omega;\eps) = \ldm_\eps (B^{Y,p}_1, \ev_{\ell^p(\Omega)})$ and $b^Y(\Omega,\eps)$ is $|\Omega|^{-1} \ldm_\eps (B^{Y,p}_1, \ev_{\ell^p(\Omega)})$.

\begin{teo}\label{tadd}
Let $\{\Omega_i\}$ be a F{\o}ner sequence for $\Gamma$ and let $Y_j \subset \ell^p(\Gamma;V_j)$, for $j=1$ or $2$, be $\Gamma$-invariant linear subspaces. Let $Y_0 = Y_1 \oplus Y_2 \subset \ell^p(\Gamma;V_1 \oplus V_2)$ be their direct sum. Then $\dlp Y_0 = \dlp Y_1 + \dlp Y_2$.
\end{teo}
\begin{proof}

\par As remarked before, $\dlp Y \leq \dlp Y_1 + \dlp Y_2$ is a rather simple consequence of the properties of $\ldm$. The rest of the proof is quite close to that of \cite[Theorem 5.1]{moi-lp}, and is consequently quite technical.

\par It is straightforward to check that the function $a$ has these four properties:
\[
\begin{array}{lll}
(\mathrm{a}) \; a \textrm{ is $\Gamma$-invariant, \ie }
          & \forall \gamma \in \Gamma,
                   & a(\eps,\gamma \Omega) = a(\eps,\Omega)\\
(\mathrm{b}) \; a \textrm{ is decreasing in $\eps$, \ie}
          & \forall \eps'\leq \eps,
                   &  a(\eps',\Omega) \geq a(\eps,\Omega) \\
(\mathrm{c}) \; a \textrm { is $K$-sublinear in $\Omega$, \ie}
          & \exists K \in \rr_{>0},
                   & a(\eps, \Omega) \leq K|\Omega|\\
(\mathrm{d}) \; a \textrm { is $c$-subadditive in $\Omega$, \ie}
          & \exists c \in ]0,1]   ,
                   & a(\eps, \Omega\cup \Omega') \leq a(c\eps, \Omega) + a(c\eps, \Omega')\\
\end{array}
\]
\par Monotonicity in $\eps$ is a direct consequence as this same property for $\ldm_\eps$ whereas $\Gamma$-invariance follow from that of $Y$. In the present case $c = 2^{-1/p}$ (as was explained at the beginning of section \ref{sadd}) and $K$ is, depending on which subspace one looks at, $\dim_\kk V_1$, $\dim_\kk V_2$ or the sum of the two (since $(B^Y,\ev_{\ell^p(\Omega)})$ always maps to $\ell^p(\Omega;V_i)$ with fibres of diameter $0$). In what follows, it is sufficient to have $K= \dim_\kk V_1 + \dim_\kk V_2$. 

\par Given a F{\o}lner sequence $\{\Omega_i\}$ there exists a subsequence $\{\Omega_i'\}$ such that $\limm{i \to \infty} b(\Omega'_i,\eps)$ converges to a decreasing function $l^Y:[0,1] \to [0,\dim_\kk V]$ such that $\limm{\eps \to 0} l^Y(\eps) = \dlp Y =: l^Y$. This is a consequence of a more general theorem, known as Helly's selection principle, concerning sequences of functions of bounded variation (\cf \cite[{\S}36.5 theorem 5, p.372]{KF}).

\par As we are dealing with a finite number of spaces, there exists a [sub]sequence $\{\Omega_i''\}$ of F{\o}lner sets such that, for $j=0,1$ or $2$, $\limm{i \to \infty} b^{Y_j}(\Omega_i'',\eps) = l^{Y_j}(\eps)$. 

\par Now, fix some $\delta \in ]0,\tfrac{1}{2}[$. Indeed, lemma \ref{paciter} ensures it is possible to find a \emph{finite} subsequence $\{F_i\}_{1 \leq i \leq n}$ of F{\o}lner sets so that  any $\Omega$ far away in the (many times refined) infinite sequence admits an $\delta$-quasi-tiling by such sets missing at most $\delta |\Omega|$ elements of $\Omega$.  It has to be stressed that $n$ does not depend on the set $\Omega$, but is valid for any big enough set in the sequence. Denote by $G_{F_j}$ the set of translates of $F_j$ obtained (some of the $G_{F_j}$ may be empty, but not all of them), and let $\Omega^{(0)}$ be the elements not covered by the quasi-tiling (so $|\Omega^{(0)}| \leq \delta |\Omega|$).

\par Hence, given $\Omega$, $\{\gamma F_i\}$ forms an $\delta$-quasi-tiling, where $1\leq i\leq n$ and $\gamma \in G_{F_i} \subset \Gamma$. Let $k = a(\Omega,\eps)$ and $L^{-k}$ be the associated linear space, and let $k_{i,\gamma}$ be the codimension of $L^{-k} + \ell^p(F_{i,\gamma}';V)$ where $F_{i,\gamma}'$ is the disjoint family extracted from the $\delta$-quasi-tiling. Then
\[
k \geq \somme{i,\gamma \in G_{F_i}}{} k_{i,\gamma} - \dim \ell^p(\Omega^{(0)};V) \geq \somme{i=1, \gamma \in G_{F_i}}{} k_i - 2 \delta K |\Omega| 
\]
where $k_i = a^{Y_0}(\gamma F_i,\eps)$. On the other hand, the $\{ \gamma F_i \}$ miss only the elements of $|\Omega^{(0)}|$, \ie
\[
\sum | \gamma F_i| \geq |\cup \gamma F_i| \geq (1-\delta) |\Omega |.
\]
This gives, using $\Gamma$-invariance (a),
\[
b^{Y_0}(\Omega,\eps) = \frac{k}{|\Omega|} 
  \geq \somme{i=1}{n} b^{Y_0}(F_i,\eps) \somme{\gamma \in G_{F_i}}{} \frac{|\gamma F_i|}{|\Omega|}  - 2 \delta K \geq \minn{1 \leq i \leq n} b^{Y_0}(F_i,\eps) (1-\delta) -2 \delta K.
\]
So the sequence of function $b^{Y_0}(\Omega_i,\eps)$ is essentially an increasing sequence (of decreasing functions). Thus, up to refining even more the sequence, it can be assumed that $b^{Y_0}(\Omega_i''',\eps) < l^{Y_0}(\eps) + \delta$.

\par Reapplying lemma \ref{paciter} to $\Omega$, a set in this F{\o}lner sequence $\{\Omega_i'''\}$, and thanks to repeated use of $c$-subadditivity (d), we have that (for $j=1,2$)
\[
\begin{array}{rll}
a^{Y_j}(\eps, \Omega)
   & \leq & \Somme{i=1}{n} \Big( \Somme{\gamma \in G_{F_i}}{} a^{Y_j} (c^{\kappa}\eps, \gamma F_i) \Big) + a^{Y_j} (c^{\kappa} \eps,\Omega^{(0)}),
\end{array}
\]
where $\kappa = n+ \somme{i=1}{n}|G_{F_i}|$. Using $\Gamma$-invariance (a), the fact that these functions are decreasing in $\eps$ (b), and the $K$-sublinear property (c), this inequality yields
\[
a^{Y_j} (\eps, \Omega) \leq \somme{i=1}{n} \Big( \somme{\gamma \in G_{F_i}}{} a^{Y_j} (c^{\kappa} \eps, F_i) \Big) +K |\Omega^{(0)}|.
\]
Lemma \ref{taddlp}, with $N= \maxx{1 \leq i \leq n} |F_i|$, gives
\[
a^{Y_1}(\eps, \Omega) + a^{Y_2}(\eps, \Omega) \leq \somme{i=1}{n} \Big( \somme{\gamma \in G_{F_i}}{} a^{Y_0} ( N^{-\Delta(p)} c^\kappa \eps, F_i) \Big) + 2 K |\Omega^{(0)}|
\]
On one hand, $|\Omega^{(0)}|\leq \delta |\Omega|$  and $\frac{ a^{Y_0} ( c^\kappa \eps, F_i)}{|F_i|} \leq l^{Y_0} (c^\kappa \eps)+\delta $. Thence,
\[
\begin{array}{rll}
\fr{a^{Y_1}(\eps, \Omega)}{|\Omega|} + \fr{a^{Y_2}(\eps, \Omega)}{|\Omega|}
  & \leq & \Somme{i,\gamma \in G_i}{} \fr{ a^{Y_0}(  N^{-\Delta(p)} c^\kappa \eps, F_i)}{|F_i|} \fr{|\gamma F_i|}{|\Omega|} +2K\fr{|\Omega^{(0)}|}{|\Omega|} \\
  & \leq & \big( l^{Y_0} (  N^{-\Delta(p)} c^\kappa \eps) + \delta \big) \Somme{i,\gamma \in G_{F_i}}{} \fr{|\gamma F_i|}{|\Omega|}+ 2K \delta
\end{array}
\]
On the other hand, the $\{ \gamma F_i \}$ are $\delta$-disjoint. Thus
\[
(1-\delta) \sum | \gamma F_i| \leq |\cup \gamma F_i| \leq |\Omega |.
\]
This shows that, $\forall \eps>0$ and for any $\Omega$ big enough,
\[
b^{Y_1}(\eps, \Omega) + b^{Y_2}(\eps, \Omega) \leq \big( l^{Y_0} (  N^{-\Delta(p)} c^\kappa ) + \delta \big) \somme{i,\gamma \in G_i}{} \fr{|\gamma F_i|}{|\Omega|}+ 2K \delta \leq \fr{  l^{Y_0} (  N^{-\Delta(p)} c^\kappa \eps ) + \delta }{1-\delta} + 2K \delta
\]
Lastly, for $j=1$ or $2$, take $\eps$ and $\Omega$ so that $l^{Y_j}  - \delta \leq b^{Y_j}(\eps,\Omega)$. This means that
\[
l^{Y_1}  + l^{Y_2}  \leq \fr{  l^{Y_0} (  N^{-\Delta(p)} c^\kappa \eps ) + \delta }{1-\delta} + 2(K+1) \delta.
\]
At this point, the only place where $\Omega$ still (indirectly) appears is in $\kappa$. However, taking $\eps \to 0$ gives that $\forall \delta \in ]0,\tfrac{1}{2} [$,
\[
\dlp Y_1 + \dlp Y_2 \leq \frac{\dlp Y_0 +\delta}{1-\delta} + 2(K+1) \delta
 \leq \dlp Y_0 + 4 \delta (K+1)
\qedhere
\]
\end{proof}

\par Reciprocity (P5) can be obtained as a simple consequence of additivity and reduction. Let us briefly recall the definition of the induced and reduced subspace.
\begin{defi}\label{dinduc}
Let $\Gamma_1 \subset \Gamma_2$ be an (amenable) subgroup and $Y_1 \subset \ell^p(\Gamma_1;V)$ a $\Gamma_1$-invariant subspace. The induced subspace $Y_2 \subset \ell^p(\Gamma_2;V)$ is
\[
Y_2 = \{ y \in \ell^p(\Gamma_2;V) \mid \forall g \in \Gamma_2 /\Gamma_1, \big( \gamma \mapsto y(\gamma g) \big) \in Y_1 \subset \ell^p(\Gamma_1;V) \}
\]
Given $Y_2 \subset \ell^p(\Gamma_2;V)$, the reduced subspace $Y_1 \subset \ell^p(\Gamma_1;V^{[\Gamma_2:\Gamma_1]})$ is 
\[
Y_1 = \big\{ y \in \ell^p(\Gamma_1;V^{[\Gamma_2:\Gamma_1]}) \mid \big( \gamma_2 \mapsto y(\gamma_1)(g) \big) \in Y_2 \textrm{ où } \gamma_2 = \gamma_1 g \big\}
\]
\end{defi}

\begin{prop}\label{tredu}
If $\Gamma_1 \subset \Gamma_2$ is a subgroup of finite index, and let $Y_2 \subset \ell^p(\Gamma_2;V)$. If $Y_1 \subset \ell^p(\Gamma_1;V^{[\Gamma_2:\Gamma_1]})$ is the reduced subspace then $[\Gamma_2:\Gamma_1] \ell^p (Y,\Gamma_2) = \ell^p (Y,\Gamma_1)$. 
\end{prop}

\begin{proof}
Let $\{\Omega^{(1)}_i \}$ be a F{\o}lner sequence for $\Gamma_1$ and let $\{\Omega^{(2)}_i \} = \{\Omega^{(1)}_i G \}$ be the corresponding F{\o}lner sequence in $\Gamma_2$. Then $(B^{Y_2}_1 , \ev_{\Omega_i^{(2)}})$ is, by construction, ``isometric'' to $(B^{Y_1}_1 , \ev_{\Omega_i^{(1)}})$.
\end{proof}

That said, P5 becomes a simple consequence of P4 and P6 (\ie theorem \ref{tadd} and proposition \ref{tredu}):

\begin{cor}\label{trec}
Let $Y_2 \subset \ell^p(\Gamma_2;V)$ be the space induced from $Y_1 \subset \ell^p(\Gamma_1;V)$. Then $\dlp (Y_2,\Gamma_2) = \dlp (Y_1,\Gamma_1)$.
 \end{cor}
\begin{proof}
Let $Y \subset \ell^p(\Gamma_1;V^{[\Gamma_2:\Gamma_1]})$ be defined by $Y =  \oplus_{g \in \Gamma_2 / \Gamma_1} Y_1$. Then $\dlp (Y:\Gamma_1) = [\Gamma_2:\Gamma_1] \dlp (Y_1:\Gamma_1)$ by additivity (P6). However, $Y$ is nothing else than the reduced space from $Y_2$, so by P4, $[\Gamma_2: \Gamma_1] \dlp (Y_2,\Gamma_2) = \dlp (Y,\Gamma_1) = [\Gamma_2:\Gamma_1] \dlp (Y_1,\Gamma_1)$.
\end{proof}
The case of infinite index subgroup would be interesting, but it definitively requires finer arguments than the above.

%-------------------------------------------------------------------------
\section{Invariance}\label{sinv}
%-------------------------------------------------------------------------

The aim of this section is to discuss more in details properties P2, namely give a counterexample in  $\ell^\infty$ and gives two cases where it holds. 

The following counter-example to P2 expands on the problems present in $\ell^\infty$. The author apologizes as this same problem was overlooked in the previous paper \cite[Corollary 3.9]{moi-lp}. Two ingredients are used in this the proof which are not mentioned in the statement: the image has to be closed (so that the inverse is bounded), and the $\ell^\infty$ should be excluded (so that functions of finite support are norm-dense).%; hopefully, it will make the ingredients of the proof more salient.

In the following example, a continuous map of finite type and closed image (in $\ell^\infty$) is shown not to satisfy P2.  

\begin{ex}\label{cexinv}
Let $c_0(\zz;\kk)$ be the space of sequences tending to $0$ at infinity and $\yperf$ be the closure of the space of periodic sequences, both sitting inside $\ell^\infty(\zz;\kk)$. As shown before, $\dl{\infty} c_0 = \dl{\infty} \yperf = \dl{\infty} \ell^\infty =  1$. Quite trivially, $c_0 \cap \yperf = \{0\}$ and both spaces are closed. Let $Y = c_0 + \yperf$. The map $f:c_0 \oplus \yperf \to Y$ is of finite type and injective. Clearly, $\dl{\infty} Y = 1$ and $f$ maps a space of $\ell^\infty$-dimension $2$ in a space of dimension $1$, a contradiction to P2.

\par The problem lies in the inverse of this map $f$. Actually, if $c_0$ is included in a separable space (such as $Y$, but not $\ell^\infty$) then there exist a projection on $c_0$. The interesting one is the map back to $\yperf$, as  upon restricting to finite subset of $\zz$, this map appears to be the trivial map. To put things briefly, the problem is related to the fact that this $\Gamma$-equivariant map will not be of any ``type'', or, in other words, to the presence in $(\ell^\infty)^*$ of elements whose support lies outside $c_0$.

\par It is possible to describe explicitly the projection from $Y$ onto $\yperf$. This can be done using invariant mean, or directly using a F{\o}lner sequence $\{ \Omega_i \}$ (to make things simple, here $\Omega_i = \{0,1, \ldots ,i\}$). Indeed, let 
\[
\mu_{k,n} (y) = \limm{i \to \infty} \frac{1}{|i|} \sum_{j \in k + n i} y(j).
\]
Then $\mu_{k,n}$ is trivial on $c_0$. To extend it to $\ell^\infty$ one requires ultrafilters, however it is well-defined on $\yperf$ (that is the above limits converge in the usual sense). Thus given an element $y$ in $c_0 + \yperf$ one can reconstruct its periodic part (in a $\Gamma$-equivariant fashion) using the whole collection of $\mu_{k,n}(y)$ (where $n \in \nn$ and $0 \leq k < n$). 
\end{ex}

% \begin{ex}\label{exsimplconv} xxx
% convolution par 111 ou 121. m.q. noyau contient un esp de dim 2 dans $\Pi_3$.
% \end{ex}

\par Before moving on, we make some additional results to achieve P2 for more general maps. The first step is to deal with inclusions. 

\begin{lem}\label{tinclusion}
If $Y \subset Y'$ then $\dlp Y \leq \dlp Y'$.
\end{lem}
\begin{proof}
This is an obvious consequence of $\ldm_\eps (X, \pnc) \leq \ldm_\eps (X',\pnc)$ when $X \subset X'$ (and $B^{f(Y)}_1 \subset B^{Y'}_1$). 
\end{proof}

As a consequence, it can now be assumed that $f(Y)= Y'$. The kernel of such maps has already received much attention. 
%under the name of analytical zero divisors (see the works of \'Elek \cite{El}, Linnell \cite{Linn}, Pape \cite{Pape} and Puls \cite{Puls} on the topic). 
Compare for example with Ceccherini-Coornaert \cite{CC} work on the Garden of Eden theorem, in particular as they use mean dimension, a concept from which $\ell^p$-dimension stems. Many subtleties will here be avoided as the basic hypothesis is that $f$ is an isomorphism. A problem which clearly arises in example \ref{cexinv} is that the inverse is not of finite type (actually not of any type).

\begin{prop}\label{tinvtypl1}
If $f:Y \to Y'$ is a $\Gamma$-invariant isomorphism whose inverse is of $\ell^1$ type, then $\dlp Y \leq \dlp Y'$.
\end{prop}

\begin{proof}
Suppose $h: B^{Y'}_1 \to \kk^k$ is a map realizing $\ldm_\eps(B^{Y'}_1,\ev_{\ell^p(\Omega)}) = k$. Then $h \circ f$ is a logical candidate to show that $\ldm$ of $B^Y_1$ must be $\leq k$. However, $ \ker h \circ f = f^{-1} \ker h$. Consequently, we are looking for a bound on $\| y \|_{\ell^p(\Omega)}$ given that $\| f(y) \|_{\ell^p(\Omega)}$ is small. Roughly, if $g:Y' \to Y$ is the inverse of $f$, then a bound on the norm of $g$ is sought.

Thus, let $a \in \ell^1(\Gamma;\homo(V_1,V_2))$ be the element realizing the convolution describing $g$. Let $A_\delta \subset \Gamma$ be such that $\|a\|_{\ell^1(A_\delta)} \geq \|a\|_{\ell^1(\Gamma)}-\delta$ (hence $\|a\|_{\ell^1(A_\delta^\comp)} \leq \delta$). Let $g_\delta$ be the convolution by the restriction of $a$ to $A_\delta$. Young's inequality gives that $\| g - g_\delta \|_{\ell^p \to \ell^p} \leq \| a \|_{\ell^1(A_\delta^\comp)} \leq \delta$. Furthermore,
\[
\|g_\delta (y') \|_{\ell^p(\Omega)} \leq \|g_\delta \|  \|y'\|_{\ell^p( \FE_{A_\delta} \Omega)} \leq (\|g\| + \delta)  \|y'\|_{\ell^p( \FE_{A_\delta} \Omega)}
\]
where $\FE_A \Omega = \Omega \cup \del_A \Omega$ is the $A$-closure of $\Omega$. Putting everything together gives
\[
\begin{array}{r@{\, \leq \,}l}
\| g(y') \|_{\ell^p(\Omega)} 
  & \| g(y') - g_\delta(y') \|_{\ell^p(\Omega)} + \|g_\delta (y') 
\|_{\ell^p(\Omega)}  \\
  & \| g(y') - g_\delta(y') \|_{\ell^p(\Gamma)} + \|g_\delta \| \cdot \|y'\|_{\ell^p( \FE_{A_\delta} \Omega)} \\
  & \delta + (\|g\| +\delta) \|y'\|_{\ell^p( \FE_{A_\delta} \Omega)} \\
\end{array}
\]
This, and the fact that $f(B^Y_1) \subset B^{Y'}_{\|f\|}$, shows that, 
\[
\ldm_{\delta + (\|g\| +\delta) \eps } (B^Y,\ev_{\ell^p(\Omega)}) \leq \ldm_{\eps /\| f\|} (B^{Y'}_1 , \ev_{\ell^p(\FE_{A_\delta} \Omega)}).
\]
Upon taking the limit along $\{\Omega_i\}$, this implies that $l^Y( \delta + (\|g\| +\delta) \eps) \leq l^{Y'}( \eps /\| f\| )$ ($l^Y$ are obtained by Helly's selection principle, see the proof of theorem \ref{tadd}). As this last inequality holds for any $\delta \in ]0,\tfrac{1}{2}[$ and $\eps > 0$, the conclusion follows.
\end{proof}
The previous result hinges only on a approximation of the map $g$ in operator norm. Note that approximations in the strong operator topology also hold if $g$ is of $\ell^p$ type (for $1<p<\infty$) and the group is amenable thanks to a result of Cowling \cite{Cow}. This might enable an extension of the previous result, but the weak operator topology seems insufficient. What is however simple, is that it extends to the norm closure (for the operator norm $Y \to Y'$) of maps of finite type, \ie pseudo-functions.

The argument of proposition \ref{tinvtypl1} requires boundedness of the inverse to construct a linear map. However, only a linear subspace is really necessary.
\begin{teo}\label{tinv-new}
Let $Y, Y' \subset \mysp(\Gamma;V)$ where $\mysp \neq \ell^\infty$. If $f:Y \to Y'$ is a $\Gamma$-invariant isomorphism of finite type, then $\dlp Y \leq \dlp Y'$.
\end{teo}
\begin{proof}
Assume that $\ldm_\eps (B^{Y'}_1,\ev_{\ell^p(\Omega)}) =k_\Omega$. Let $L'$ be a linear subspace of $Y'$ such that $\big( L'+\ell^p(\Omega^\comp) \big)\cap \img f$ is of codimension $k$ (in $\img f$) and $\forall y \in B^{Y'}_1 \cap L', \|y\|_{\ell^p(\Omega)} < \eps$. Since $f$ is injective and linear, there exists a linear subspace $L \subset Y$ such that $f(L) = L'$. Suppose $F$ is the support of the convolution defining $f$. If for all large enough sets $\Omega$ in a F{\o}lner sequence $\ker f \cap (Y + \ell^p(\Omega^\comp))$ is non-empty, then (unless $\mysp = \ell^\infty$) $\ker f \cap Y \neq \vide$. Thus, for $\Omega$ large enough (and $\IN_A \Omega = \Omega \setminus \del_A \Omega$, $f$ has no kernel in $L + \ell^p\big ((\IN_A \Omega)^\comp \big)$, and $L$ will be our candidate with codimension less than $k_\Omega + |\del_A \Omega|  \dim V$.

To show the norm estimate on such elements, a (non-linear) map will come handy.
%a (non-linear) map $g$ will be defined on $\ell^p$ so that $g(L \cap B^Y_1) = L' \cap B^{Y'}_1$ and if $\|g(x)\|_{\ell^p(\Omega)} < \eps$ then $\|x\|_{\ell^p(\Omega')} < \eps$. 
Let $g(0)=0$ and, for $y \neq 0$, $g(y) = \frac{f(y)}{\|f(y)\|} \|y\|$. This maps sends $L \cap B^Y_1$ to $L' \cap B^{Y'}_1$. So if $y \in L \cap B^Y_1$ then $g(y) \in L'\cap B^{Y'}_1$ and so $\|g(y)\|_{\ell^p(\Omega)} < \eps$. Furthermore, $y$ can be assumed to be supported on $\ell^p(\IN_A \Omega)$, and so,
\[
\eps > \frac{\|f(y)\|_{\ell^p(\Omega)}}{\|f(y)\|} \|y\| = \|y\|_{\ell^p(\FE_A \Omega)}.
\]
Thus, $\ldm_\eps (B^Y_1, \ev_{\ell^(\IN_A \Omega)}) \leq  \ldm_\eps (B^Y_1, \ev_{\ell^(\IN_A \Omega)}) + |\del_A \Omega| \dim V$, which yields the conclusion.
\end{proof}
There are continuous linear maps of finite type without inverse of $\ell^1$ type (or pseudo-functions) which means that theorem \ref{tinv-new} adds to proposition \ref{tinvtypl1}. It is not obvious to the author if there are continuous linear maps $f$ whose inverse $f^{-1}$ is a pseudo-function and who are not themselves of finite or $\ell^1$ type.

\section{Further remarks and questions}\label{sfin}
%-------------------------------------------------------------------------

Question \ref{laquestion} admits a very simple answer in the case of $\ell^1$ (actually, for any group $\Gamma$), due to the peculiar behaviour of increasing sequences whose union is norm dense.
\begin{rem}\label{repl1}
Indeed, from the fact that $\srl{\cup Y_k} = \ell^1$, there is a also sequence $y_k \in Y_k$ which tends (with $\ell^1$-convergence) to $\delta_e$. Let $\eps_k = \|\delta_e-y_k\|_{\ell^1}$. This sequence allows a uniform weak approximation of the elements of $\ell^1$ by elements in the $Y_k$. Indeed, let $z \in \ell^1$, then the convolution $z * y_k$ lies in $Y_k$ (by closedness of $Y_k$, $\Gamma$-invariance of $Y_k$, norm density of $\delta_e$ and its translates in $\ell^1$ and Young's inequality). Furthermore, $\forall \alpha \in \ell^\infty$, 
\[
|\pgen{ \alpha , z * y_k} - \pgen{ \alpha, z}| = |\pgen{ \alpha , z * (y_k - \delta_e)}| \leq \|\alpha\|_{\ell^\infty} \|z * (y_k - \delta_e)\|_\ell^1 \leq \eps_k \|\alpha\|_{\ell^\infty} \| z\|_\ell^1.
\]
Take $y_k^* \in Y_k^\perp$, $\|y_k^*\|_{\ell^\infty} = \sup_{\|z\|_{\ell^1} = 1} |\pgen{y_n^*,z}|$. Taking some $z'$ of norm $1$ (in $\ell^1$) which is $\eps_k$-close to the supremum and using the above-mentioned uniform weak approximation $\tilde{y}_k = z' * y_k$,
\[
(1-2\eps_k)\|y_k^*\|_{\ell^\infty} \leq |\pgen{y_k^* , \tilde{y}_k}| =  0, 
\]
since $\pgen{y_k^* , \tilde{y}_k} =0$. Consequently, $y_k^* =0$ as soon as $k$ is large enough. This implies that there is actually a $k$ for which $Y_k = \ell^1$ and in particular that $X \cap Y_k \neq \{0\}$. \hfill $\Diamond$
\end{rem}

It would be nice if this argument could be adapted to answer question \ref{laquestion} in $\ell^p$ with $p \in ]1,\infty[$. Indeed, it seems somehow strange to have a sequence $y_k \to_k \delta_e$ (where $y_k \in Y_k$) and $x_{n,k}^\perp + y_{n,k}^\perp \to_n \delta_e$ (where $x_{n,k}^\perp \in X^\perp$ and $y_{n,k}^\perp \in Y_k^\perp$) while $\cap Y_k^\perp =\{0\}$. The problem lies in the fact that, though the norm of $x_{n,k}^\perp + y_{n,k}^\perp$ is (obviously) bounded, the author failed to find a reason forcing the norm of $x_{n,k}^\perp$ to be bounded. If this would be the case, then question \ref{lareponse} would have an easy answer.

The first questions concerns the properties that are not determined here.

\begin{ques}\label{qp10}
Does P8' hold for some $p\neq 2$? Does P10 holds for $p =1$, \ie are closed spaces of full $\ell^1$-dimension of finite codimension?
% Does P10' holds for $p \in [1,2]$, $\mydim = \dlp$ and $\mysp = \ell^p$, \ie are closed spaces of full $\ell^p$-dimension of finite codimension? More ambitiously, if further $p\neq 1$, is the only closed space of full $\ell^p$-dimension $\ell^p(\Gamma;V)$? 
\end{ques}

% \par As a consequence to this question, one would make a significant step in extending theorem \ref{lareponse} to all $p$ except $p=\infty$. In view of proposition \ref{tcomp}, note that all $\Gamma$-invariant closed subspace $Y$ such that $B^Y_1$ is weak$^*$ dense in $B^{\ell^p}_r$ (for some $r \in ]0,1[$) are of finite codimension (a simple consequence of Hahn-Banach theorem). 

\par It seems quite probable that P9' holds for $\ddl{\infty}$ in $c_0$ even though P8' is false for $\dl{1}$ in $\ell^1$. The point being that for a space $Y \subset c_0$, its annihilator $Y^\perp$ is weak$^*$-closed. As such, this excludes the spaces constituting the main counterexample to P8'. In other words,

\begin{ques}\label{qp9'}
If $Y_n$ is a weak$^*$-closed sequence of decreasing subspaces in $\ell^1$ such that $\cap Y_n = \{0\}$, does $\dl{1} Y_n \to 0$?
\end{ques}

\par Indeed, the dimension constructed here being defined by looking at increasing sequences of finite dimensional spaces, it should behave correctly only in a weak sense. %This seems also coherent with the fact that (much like the usual von Neumann dimension) it should be thought of as a the measure of the annihilator in the Pontryagin dual of the group (when the group is Abelian).

Thanks to a conversation with B.Nica, the author realised there is an strong obstruction to positivity for $p>2$.
\begin{rem}\label{paspos}
%If positivity was to be true, similar arguments as in section \ref{scohom} could be used to show that the algebraic zero divisor conjecture holds if and only if the $\ell^p$ analytic version holds (see \cite{Linn}). 
It is not too hard to devise arguments which show that if P7 is true, then the algebraic zero divisor conjecture holds if and only if the $\ell^p$ analytic version holds (see \cite{Linn}). Indeed, if $\dlp Y$ is positive, then it is possible to construct an element which vanishes on the boundary of a F{\o}lner set. Since it holds in $\ell^2$ by \cite{El-zd}, it also holds in $\ell^p$ for $p \in [1,2]$. So given the actual properties, there is no new result obtained. However, if one would have positivity for $p> 2$, then there would be a contradiction, as it is known (see \cite{Puls-zd1}) that it does not even hold for the Abelian groups $\zz^k$ (where $k$ is large enough). \hfill $\Diamond$
\end{rem}

% \par It would seem strange (but by no way impossible) that there exists two distinct notions of dimension (though 
% %G.~\'Elek pointed out to the author that 
% for $p=2$ there uniqueness is known). Given that the properties of $\dlp$ and $\ddlp$ are closely related, a positive answer to the following question would make things much easier. %(in fact, only questions \ref{qp10} and \ref{qp9'} would remain to have all the properties either established or disproved).
% 
% \begin{ques}\label{q:egal}
% Is $\ddlp$ equal to $\dlp$?
% \end{ques}
% 

\par Given the recent work of B.Hayes \cite{Hay}, another valuable question is to ask whether the two definitions are equal. Furthermore, it might be true that Hayes' $\ell^p$ dimension coincide with $\ddlp$. Here is another small step in this direction. In the paper of Ioffe and Tikhominov \cite[\S 4]{IT} one sees 4 (more or less) classical notions of width coming up. In the present formulation, they could be rephrased, for $X$ a subset of a (pseudo-)normed linear space, as 
\[
\begin{array}{ll}
\mathrm{bdim}_\eps X = \sup \{ k \mid \exists L^{k}, B_\eps \cap L^k \subset X \} &
\ldm_\eps X =  \inf \{ k \mid \exists L^{-k}, L^{-k} \cap X \subset B_\eps \}  \\
\mathrm{tdim}_\eps X = \sup \{ k \mid \exists L^{-k}, L^{-k} + X \supset B_\eps \} &
\mathrm{cdim}_\eps X = \inf \{ k \mid \exists L^{k}, B_\eps + L^k \supset X \} 
\end{array}
\]
where $L^k$ denotes a subspace of dimension $k$ and $L^{-k}$ a subspace of codimension $k$. Changing $\ldm$ by one of the other three numbers does not seem change things so drastically (though, if in some cases sub-additivity is straightforward with the other it might become sup-additivity). The above question (using polars) would probably follow if these four numbers agree. Some inequalities are relatively easy to get: $\mathrm{bdim}_{2\eps} X \leq \ldm_\eps X \leq \mathrm{cdim}_{\eps/2} X$.
%With expectedly more difficulty, could the $\ell^p$-dimension as defined by Hayes be re-expressed in one of the above, (\eg in terms of $\mathrm{cdim}$)?

\begin{ques}\label{q:egal}
Does $\ddlp$ coincide with Hayes' $\ell^p$ dimension (see \cite{Hay}) and/or by the replacement of $\, \ldm_\eps$ by one of the above widths?
\end{ques}

\par In $\ell^p$ there are uncomplemented spaces. However the examples of such spaces are quite irregular (for a simplified proof of Sobczyk's result, see Tomczak-Jaegermann's book \cite[p.252]{TJ}). Thence:
\begin{ques}
Are all $\Gamma$-invariant closed subspaces $Y \subset \ell^p(\Gamma;V)$ complemented? 
\end{ques}
\par Many partial answers exist. %J.E.~Gilbert \cite{Gil} shows, in particular, that there are projections on weak$^*$ closed translation subspaces in $\ell^\infty$ for discrete abelian groups. 
Among many, note that Liu, van Rooij and Wang \cite{LRW} and Bekka \cite{Bek} exhibit an intimate connection between the existence of projections and the presence of bounded approximate identities. Such results would be interesting in $\ell^p$. %Perhaps lemma \ref{contl1} is a good starting point. 
H.P.~Rosenthal \cite{Ros} basically solves the question for the Abelian case if $1<p<2$: $A$ is complemented in $L^p(\Gamma)$ if and only if $A$ is the set of all elements whose Fourier transforms vanish almost everywhere outside of some measurable set $E$ of $\hat \Gamma$, and $\widehat{L^p(\Gamma)}$ is closed under (pointwise) multiplication by the characteristic function $\chi_E$. Note also that, by \cite[Lemma 3.1]{Ros} if there is a projection and $1<p<\infty$, there is automatically a $\Gamma$-equivariant projection ($\Gamma$ arbitrary). %P.J.~Wood \cite[Corollary 6 and Theorem 2]{Wood} gives condition (which are automatically satisfied by amenable discrete groups) for subspaces of the Fourier algebra. %\cite{AM}, \cite{Kep}, \cite{Rblt}.

%\par As a corollary of a positive answer to the above question, one could try to extract P8 (resp. P9) from P8' (resp. P9'). More importantly, the question \ref{laquestion} would have a much more direct answer.

% % %\par In fact, this simpler question has a good chances of admitting a negative answer, as maps of $c_0$ type fit in the theory of summability (see \cite[Remark 9.3.9 and examples 6.1.4, 6.4.10 and 7.2.4]{Edw} or \cite{Zel}).

\par A proof of additivity would be much simpler if one could have something in the flavour of lemma \ref{taddl2} (which in some sense only involves a decomposition in a direct sum) for any $p$.

\begin{ques}
For which $p\in [1,\infty]$ (if any outside $p=2$) does there exist a constant $ C>1$ such that given $Y_1$,$Y_2 \subset \ell^p$ of finite dimension, $\forall L \subset Y_1 \oplus Y_2$, $\exists L_i \subset L$ and a map $\pi_i: L_i \to Y_i$ satisfying

(a) $L_1 \cap L_2 = \{0\}$ and $L_1 + L_2 = L$.

(b) $\forall l \in L_i, \tfrac{1}{C} \| \pi_i l \|_p \leq \|l\|_p \leq C \| \pi_i l\|_p$.
\end{ques}

%More precisely, the recent paper of B.Hayes \cite{Hay} points there might be other simplifications possible.

%It would be tempting for $p \in [1,2]$ to define the dimension of a $Y \subset \ell^p$ by looking at the von Neumann dimension of its closure. This would unfortunately lead to a dimension defined for any group, and there are foreseen contradiction for the existence of such an object. Even if one keeps astray from properties involving limits of space, the problem is that $Y \subset \ell^2(\Gamma;V)$ is not closed. As a consequence, it does not make sense to speak of what is added to $Y$ by the closure in a dimensional sense: the quotient $\srl{Y}/Y$ is not well defined.  

Finally, there would be a somehow more intuitive (and probably much weaker) $\ell^p$-dimension if the aim is to answer \ref{laquestion}. Remember that to define von Neumann dimension of a closed $\Gamma$-invariant subspace $X \subset \ell^2(\Gamma)$, simply take $P_X$ the projection on $X$. Then $P_X$ is $\Gamma$-equivariant, \ie $\gamma \cdot (P_X f) = P_X (\gamma \cdot f)$ for any $f \in \ell^2(\Gamma)$. Define the von Neumann dimension of $X$ to be $\tau(P_X) = \pgen{\delta_e,P_X \delta_e}$. Four facts are needed to conclude: 
\begin{enumerate}\renewcommand{\labelenumi}{{\normalfont T\arabic{enumi}}}
\item $\tau(P_X) \in [0,1]$ (since $P_X \delta_e$ is the closest point to $\delta_e$ in $X$ and projections reduce norm).
\item $\tau(P_{Y_k}) \to 1$ if $Y_k$ is an increasing sequence of subspaces whose reunion is dense (as one gets always closer to $\delta_e$, \ie there is a sequence $y_k \in Y_k$ with $y_k \to \delta_e$).
\item $\tau(P_X)=0$ if and only if $X= \{0\}$. 
\item If $X \cap Y = \{0\}$ then the direct sum $X + Y$ (inside $\ell^2(\Gamma)$) has dimension $\tau(P_{X + Y}) = \tau(P_X) + \tau(P_Y)$.
\end{enumerate}
Thus, if $Y_k \cap X = \{0\}$ for all $k$, the spaces $X + Y_k$ get eventually of dimension bigger than $1$, a contradiction. Thus, there is a $k$ for which $Y_k \cap X \neq \{0\}$. 

Now, T1 and T2 are very straightforward. For T3, observe that the Dirac mass and its translates are dense in $\ell^2(\Gamma)$ and $P_X$ is $\Gamma$-equivariant, so $P_X \delta_e =0$ if and only if $X= \{0\}$. Again, as $P_X \delta_e$ is the closest point to $\delta_e$ in $X$, $\tau(P_X)$ must be positive. This motivates the following definition. Le $X$ be a closed $\Gamma$-invariant subspace of $X \subset \ell^p(\Gamma)$ where $1<p<\infty$. There is not necessarily a projection on $X$. However, there is a trick: Let $B^X_1$ be the unit ball in $X$, then there is a nearest point projection $P_X:B^{\ell^p}_1 \to B^X_1$ (see \cite[Section 2.2, page 40]{BL}). Let $D(X) = \pgen{ \delta_e , P_X \delta_e}$. The three easy facts are obtained again (with basically the same arguments):
\begin{enumerate}\renewcommand{\labelenumi}{{\normalfont D\arabic{enumi}}}
\item $D(X) \in [0,1]$.
\item $D(Y_k) \to 1$ if $Y_k$ is an increasing sequence of subspaces whose reunion is dense.
\item $D(X)=0$ if and only if $X= \{0\}$. 
\end{enumerate}
One also has very cheaply that $D(X) = 1$ if and only if $X = \ell^p(\Gamma)$. So this brings the following
\begin{ques}
Does there exist a function $f:[0,1] \times [0,1] \to [0,2]$ such that, if $X \cap Y = \{0\}$ then $f(D(X),D(Y)) \leq f(D(X+Y))$ and $\forall x \in ]0,1]$ there is an $\eps>0$ such that $f(x,1-\eps) > 1$?
\end{ques}
Obviously, since the definition holds for any group, it can not have normal additivity (see introduction of \cite{moi-lp}). But some non-linear monotonicity as above is sufficient to get an answer to question \ref{lareponse} (and is insufficient to contradict anything). Note further that $\tau(P_X) = \|P_x \delta_e\|^2_{\ell^2}$. Lemma \ref{tracelp} could be of help to study $D(X)$. One could define alternatively define the ``dimension'' as $N(X) = \|P_X \delta_e\|^p_{\ell^p}$. This is probably not much of a change since, lemma \ref{tracelp} relates these two quantities by $N(X) = D(X)^p + (1-D(X))^{p-1} D(X)$.

%\vspace*{1cm}
%
%quelques refs à mon art prec:
%
%\begin{tabular}{l@{\quad $\to$ \quad}l}
%pos, p=1         & Prop 4.1 \\
%c-ex à cont, p=1 & Ex 4.2  \\
%diml2=dim vN     & cor A.2 \\
%indep sur Folner & Cor 5.2 \\
%OW gen			  & Thm 5.1
%\end{tabular}

% \medskip
% % The data information below will be filled by AIMS editorial staff
% Received xxxx 20xx; revised xxxx 20xx.
% \medskip

\end{document}